\documentclass{amcjoucc}

\usepackage[small,bf,hang]{caption} 

\usepackage{tikz}

\usepackage{hyperref}

\usepackage{color}

\newtheorem{observation}[thm]{Observation}

\graphicspath{{figures/}}

\newenvironment{dedication}
{\begin{quotation}\begin{center}\begin{em}}
{\end{em}\end{center}\end{quotation}}

\begin{document}

\begin{frontmatter}   

\titledata{On Hypohamiltonian Snarks\\ and a Theorem of Fiorini}{}          

\authordata{Jan Goedgebeur}            
{Department of Applied Mathematics, Computer Science \& Statistics, Ghent University, Krijgslaan 281-S9, 9000 Ghent, Belgium}    
{jan.goedgebeur@ugent.be}                     
{The first author is supported by a Postdoctoral Fellowship of the Research Foundation Flanders (FWO).}          

\authordata{Carol T.\ Zamfirescu}            
{Department of Applied Mathematics, Computer Science \& Statistics, Ghent University, Krijgslaan 281-S9, 9000 Ghent, Belgium}    
{czamfirescu@gmail.com}
{The second author is a PhD fellow at Ghent University on the BOF (Special Research Fund) scholarship~01DI1015.}                                       


\begin{dedication}
In memory of the first author's little girl Ella.
\end{dedication}

\begin{abstract}
We discuss an omission in the statement and proof of Fiorini's 1983 theorem on hypohamiltonian snarks and present a version of this theorem which is more general in several ways. Using Fiorini's erroneous result, Steffen showed that hypohamiltonian snarks exist for some $n \ge 10$ and each even $n \ge 92$. We rectify Steffen's proof by providing a correct demonstration of a technical lemma on flower snarks, which might be of separate interest. We then strengthen Steffen's theorem to the strongest possible form by determining all orders for which hypohamiltonian snarks exists. This also strengthens a result of M\'{a}\v{c}ajov\'{a} and \v{S}koviera. Finally, we verify a conjecture of Steffen on hypohamiltonian snarks up to 36 vertices.
\end{abstract}

\end{frontmatter}   

\section{Introduction}

A graph $G$ is \emph{hypohamiltonian} if $G$ itself is non-hamiltonian, but for every vertex $v$ in $G$, the graph $G - v$ is hamiltonian. A \emph{snark} shall be a cubic cyclically 4-edge-connected graph with chromatic index~4 (i.e.\ four colours are required in any proper edge-colouring) and girth at least~5. We refer for notions not defined here to \cite{St01} and \cite{Di10}.

Motivated by similarities between the family of all snarks and the family of all cubic hypohamiltonian graphs regarding the orders for which such graphs exist, Fiorini~\cite{Fi83} studied the hypohamiltonian properties surrounding Isaacs' so-called ``flower snarks''~\cite{Is75} (defined rigorously below). The a priori surprising interplay between snarks and hypohamiltonian graphs has been investigated extensively---we now give an overview. Early contributions include Fiorini's aforementioned paper~\cite{Fi83}, in which he showed that there exist infinitely many hypohamiltonian snarks. (In fact, according to M\'{a}\v{c}ajov\'{a} and \v{S}koviera~\cite{MS07}, it was later discovered that a family of hypohamiltonian graphs constructed by Gutt~\cite{Gu77} includes Isaacs' snarks, thus including Fiorini's result.)

Skupie\'n showed that there exist exponentially many hypohamiltonian snarks~\cite{Sk07}, and Steffen~\cite{St01} proved that there exist hypohamiltonian snarks of order~$n$ for every even $n \ge 92$ (and certain $n < 92$)---we will come back to this result in Section~\ref{sect:theorem_steffen}. For more references and connections to other problems, see e.g.~\cite{MS07,Sk07,BGHM13,steffen20151}. Hypohamiltonian snarks have also been studied in connection with the famous Cycle Double Cover Conjecture~\cite{BGHM13} and Sabidussi's Compatibility Conjecture~\cite{FH09}. 

The smallest snark, as well as the smallest hypohamiltonian graph, is the famous Petersen graph. Steffen~\cite{St98} showed that every cubic hypohamiltonian graph with chromatic index~4 is \emph{bicritical}, i.e.\ the graph itself is not 3-edge-colourable but deleting any two distinct vertices yields a 3-edge-colourable graph. Nedela and \v{S}koviera~\cite{NS96} proved that every cubic bicritical graph is cyclically 4-edge-connected and has girth at least~5. Therefore, every cubic hypohamiltonian graph with chromatic index~4 must be a snark.

This article is organised as follows. In Section~\ref{sec:fiorini} we discuss the omission in Fiorini's theorem on hypohamiltonian snarks~\cite{Fi83} and its consequences and state a more general version of this theorem. In Section~\ref{sect:theorem_steffen} we first rectify a proof of Steffen on the orders for which hypohamiltonian snarks exist which relied on Fiorini's theorem---this erratum is based on giving a correct proof of a technical lemma concerning flower snarks, which may be of separate interest. We then prove a strengthening of Steffen's theorem, which is best possible, as all orders for which hypohamiltonian snarks exist are determined. Our result is stronger than a theorem of M\'{a}\v{c}ajov\'{a} and \v{S}koviera~\cite{MS06} in the sense that our result implies theirs, while the inverse implication does not hold. Finally, in Section~\ref{sect:conjecture_steffen} we comment upon and verify a conjecture of Steffen on hypohamiltonian snarks~\cite{steffen20151} for small hypohamiltonian snarks.

\section{Fiorini's Theorem Revisited}
\label{sec:fiorini}

We call two edges independent if they have no common vertices.
Let $G$ and $H$ be disjoint connected graphs on at least 6~vertices. Consider $G' = G - \{ ab, cd \}$, where $ab$ and $cd$ are independent edges in $G$, $H' = H - \{ x,y \}$, where $x$ and $y$ are adjacent cubic vertices in $H$, and let $a',b'$ and $c',d'$ be the other neighbours of $x$ and $y$ in $H$, respectively. Then the \emph{dot product} $G \cdot H$ is defined as the graph $$(V(G) \cup V(H'), E(G') \cup E(H') \cup \{ aa', bb', cc', dd' \}).$$ Two remarks are in order. (1) Under the above conditions, the dot product may be disconnected. (2) In fact, there are two ways to form the dot product for a specific $ab, cd, xy$. One also has the possibility to add the edges $\{ ac', bd', ca', db' \}$. However, in this paper we will always construct the dot product by adding the former set of edges, summarising this in the following statements as ``$a, b, c, d$ are joined by edges to the neighbours of $x$ and $y$, respectively''.

According to Skupie\'n~\cite{Sk16}, the dot product was introduced by Adel'son-Vel'ski\v{\i} and Titov~\cite{AT73}, and later and independently by Isaacs~\cite{Is75}. Its original purpose was to obtain new snarks by combining known snarks. Isaacs was the first to explicitly construct infinitely many snarks~\cite{Is75}. Fiorini then proved that the dot product can also be used to combine two \emph{hypohamiltonian} snarks into a new one. Unfortunately, Fiorini's argument is incorrect. We shall discuss this omission within this section, and correct the proof of a lemma of Steffen~\cite{St01} which depended on Fiorini's result in Section~\ref{sect:theorem_steffen}.

In a graph $G$, a pair $(v,w)$ of vertices is \emph{good} in $G$ if there exists a hamiltonian path in $G$ with end-vertices $v$ and $w$. Two pairs of vertices $((v,w),(x,y))$ are \emph{good} in $G$ if there exist two disjoint paths which together span $G$, and which have end-vertices $v$ and $w$, and $x$ and $y$, respectively.

\begin{theorem}[Fiorini, Theorem~3~in~\cite{Fi83}]\label{theorem:fiorini}
Let $G$ be a hypohamiltonian snark having two independent edges $ab$ and $cd$ for which\\
(i) each of $(a,c), (a,d), (b,c), (b,d), ((a,b),(c,d))$ is good in $G$;\\
(ii) for each vertex $v$, exactly one of $(a,b), (c,d)$ is good in $G - v$.\par
If $H$ is a hypohamiltonian snark with adjacent vertices $x$ and $y$, then the dot product $G \cdot H$ is also a hypohamiltonian snark, where $ab$ and $cd$ are deleted from $G$, $x$ and $y$ are deleted from $H$, and vertices $a, b, c, d$ are joined by edges to the neighbours of $x$ and $y$, respectively.
\end{theorem}

Cavicchioli et al.~\cite{CMRS03} point out the omissions in Theorem~\ref{theorem:fiorini}: in order for the proof to work, the given vertex pairs need to be good in the graph \emph{minus the edges} $ab, cd$. They restate the theorem and give a new proof.

\begin{theorem}[Cavicchioli et al., Theorem~3.2 in~\cite{CMRS03}]\label{theorem:cav}
Let $G$ be a hypohamiltonian snark having two independent edges $ab$ and $cd$ for which\\
(i) each of $(a,c), (a,d), (b,c), (b,d), ((a,b),(c,d))$ is good in $G - \{ ab, cd \}$;\\
(ii) for each vertex $v$, each of $(a,b), (c,d)$ is good in $G - \{ v, ab, cd \}$.\par
If $H$ is a hypohamiltonian snark with adjacent vertices $x$ and $y$, then the dot product $G \cdot H$ is also a hypohamiltonian snark, where $ab$ and $cd$ are deleted from $G$, $x$ and $y$ are deleted from $H$, and vertices $a, b, c, d$ are joined by edges to the neighbours of $x$ and $y$, respectively.
\end{theorem}

In above statements, the fact that the dot product of snarks is itself a snark had already been shown~\cite{AT73,Is75}, so indeed only the hypohamiltonicity was to be proven.

We point out that the hypotheses in Theorem~\ref{theorem:cav} are unattainable for $v \in \{ a,b,c,d \}$, which is tied to the fact that the requirements in (ii) are stronger than what is needed to prove the statement.

In~\cite{GZ}, the following (second) restatement of Theorem~\ref{theorem:fiorini} is given. Note that in \cite{GZ} the graphs are required to be cubic and below we do not state this requirement---we do however need the two vertices which are removed to be cubic. This allows us to use exactly the same proof as in~\cite{GZ}.

\begin{theorem}[Goedgebeur and Zamfirescu]\label{theorem:gz}
Let $G$ be a non-hamiltonian graph having two independent edges $ab$ and $cd$ for which\\
(i) each of $(a,c), (a,d), (b,c), (b,d), ((a,b),(c,d))$ is good in $G - \{ ab, cd \}$;\\
(ii) for each vertex $v$, at least one of $(a,b)$ and $(c,d)$ is good in $G - \{ v, ab, cd \}$.\par
If $H$ is a hypohamiltonian graph with cubic adjacent vertices $x$ and $y$, then the dot product $G \cdot H$ is also a hypohamiltonian graph, where $ab$ and $cd$ are deleted from $G$, $x$ and $y$ are deleted from $H$, and vertices $a, b, c, d$ are joined by edges to the neighbours of $x$ and $y$, respectively.\par
If $G$ and $H$ are planar, and $ab$ and $cd$ lie on the same facial cycle, then the dot product can be applied such that $G \cdot H$ is planar, as well. If $g$ ($h$) is the girth of $G$ ($H$), then the girth of $G \cdot H$ is at least $\min \{ g, h \}$. If $G$ and $H$ are cubic, then so is $G \cdot H.$
\end{theorem}

Note that the fact that $G$ is non-hamiltonian together with condition (ii) implies that $G$ must be hypohamiltonian.

In the following, we will call the pair of edges $ab, cd$ from the statement of Theorem~\ref{theorem:gz} \emph{suitable}. The Petersen graph is the smallest snark, and the two Blanu\v{s}a snarks on 18~vertices are the second-smallest snarks. All three graphs are also hypohamiltonian. Due the huge automorphism group of the Petersen graph, it can be verified by hand that it does not contain a pair of suitable edges. Although both Blanu\v{s}a snarks are dot products of two Petersen graphs, the Petersen graph does not contain a pair of suitable edges. Thus, in a certain sense, Theorem~\ref{theorem:gz} is \emph{not} ``if and only if'', i.e.\ there exist dot products whose factors do not contain suitable edges.

Let us end this section with a remark which may prove to be useful in other applications. Throughout its statement and proof, we use the notation from Theorem~\ref{theorem:gz}.

\begin{observation}
We have that $G \cdot H + ab$, $G \cdot H + cd$, and $G \cdot H + ab + cd$ are hypohamiltonian, as well.
\end{observation}

\begin{proof}
Put $N(x) = \{ a', b', y \}$ and $N(y) = \{ c', d', x \}$ such that the unique neighbour of $a'$ ($b'$, $c'$, $d'$) in $G$ is $a$ ($b$, $c$, $d$). Assume $G \cdot H + ab + cd$ contains a hamiltonian cycle ${\mathfrak h}$. Thus, at least one of $ab$ and $cd$ lies in ${\mathfrak h}$, say $ab$. We treat $H - \{ x,y \}$ as a subgraph of $G \cdot H$. If $aa', bb' \in E({\mathfrak h})$, then ${\mathfrak h} \cap H \cup a'xb' \cup c'yd'$ gives a hamiltonian cycle in $H$, a contradiction. If $aa', bb' \notin E({\mathfrak h})$, then the cycle ${\mathfrak h} \cap G + cd$ yields a contradiction. So w.l.o.g.\ $aa \in E({\mathfrak h})$ and $bb' \notin E({\mathfrak h})$. This implies the existence of a hamiltonian path in $H - \{ x,y \}$ with end-vertices $a'$ and $u \in \{ c', d' \}$. But this path together with $uyxa'$ is a hamiltonian cycle in $H$, a contradiction. It follows that $G \cdot H + ab$ and $G \cdot H + cd$ are non-hamiltonian, as well.
\end{proof}

\section{On a Theorem of Steffen on Hypohamiltonian Snarks}
\label{sect:theorem_steffen}

\subsection{Rectifying Steffen's proof}
\label{subsect:fix_steffen}

A snark is \emph{irreducible} if the removal of every edge-cut which is not the set of all edges incident with a vertex yields a 3-edge-colourable graph. Steffen's article~\cite{St01} is motivated by the following problem.

\begin{problem}[Nedela and \v{S}koviera~\cite{NS96}]\label{problem:NS}
For which even number $n \ge 10$ does there exist an irreducible snark of order~$n$? In particular, does there exist an irreducible snark of each sufficiently large order?
\end{problem}

Steffen settled the second question of Problem~\ref{problem:NS} by giving the following main result from~\cite{St01}.

\begin{theorem}[Steffen, Theorem 2.5 in~\cite{St01}]\label{Steffen}
There is a hypohamiltonian snark of order $n$\\
(1) for each $n \in \{ m : m \ge 64 \ and \ m \equiv 0 \ {\rm mod} \ 8 \}$\\
(2) for each $n \in \{ 10, 18 \} \cup \{ m : m \ge 98 \ and \ m \equiv 2 \ {\rm mod} \ 8 \}$\\
(3) for each $n \in \{ m : m \ge 20 \ and \ m \equiv 4 \ {\rm mod} \ 8 \}$\\
(4) for each $n \in \{ 30 \} \cup \{ m : m \ge 54 \ and \ m \equiv 6 \ {\rm mod} \ 8 \}$\\
(5) for each even $n \ge 92$.
\end{theorem}

Isaacs' \emph{flower snark} $J_{2k+1}$, see~\cite{Is75}, is the graph $$\left( \{ a_i, b_i, c_i, d_i \}_{i = 0}^{2k}, \{ b_ia_i, b_ic_i, b_id_i, a_ia_{i+1}, c_id_{i+1}, d_ic_{i+1} \}_{i = 0}^{2k} \right),$$ where addition in the indices is performed modulo $2k + 1$.

However, the proof of~\cite[Lemma 2.3]{St01}, which is essential for the proof of the theorem, is erroneous, since it uses Theorem~\ref{theorem:fiorini} (and it does not, by sheer coincidence, work with Theorem~\ref{theorem:gz}). We here give a correct proof of that lemma.

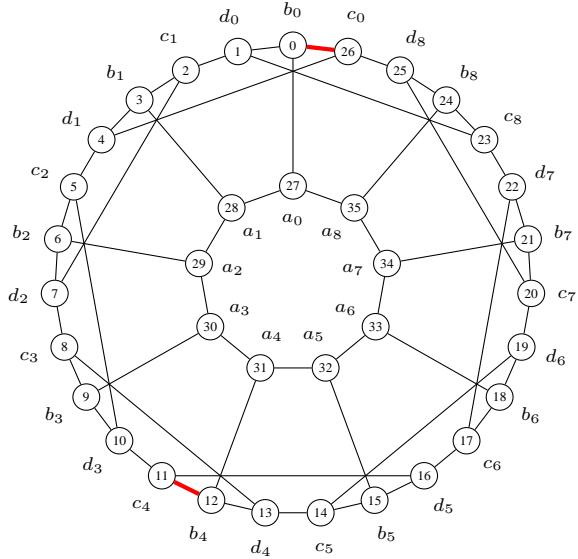
\begin{figure}[h!t]
	\centering

\begin{tikzpicture}[every node/.style={circle, draw, inner sep=1pt,
minimum size=10pt},every edge/.style={draw},scale=0.5]
    \node (0) at (0.000000000000000,6.25000000000000) {\tiny 0};
    \node (1) at (-1.44134919214025,6.08153044112390) {\tiny 1};
    \node (2) at (-2.80499487625289,5.58520400202133) {\tiny 2};
    \node (3) at (-4.01742256054087,4.78777776949361) {\tiny 3};
    \node (4) at (-5.01326995471902,3.73224119814241) {\tiny 4};
    \node (5) at (-5.73885066800171,2.47549853774473) {\tiny 5};
    \node (6) at (-6.15504845632630,1.08530111041832) {\tiny 6};
    \node (7) at (-6.23942598919543,-0.363405180690475) {\tiny
7};
    \node (8) at (-5.98743445197181,-1.79252020444431) {\tiny
8};
    \node (9) at (-5.41265877365274,-3.12500000000000) {\tiny
9};
    \node (10) at (-4.54608525983156,-4.28901023667958) {\tiny
10};
    \node (11) at (-3.43443111294254,-5.22179882133085) {\tiny
11};
    \node (12) at (-2.13762589578543,-5.87307887991193) {\tiny
12};
    \node (13) at (-0.725580713282687,-6.20773973588714) {\tiny
13};
    \node (14) at (0.725580713282685,-6.20773973588714) {\tiny
14};
    \node (15) at (2.13762589578543,-5.87307887991193) {\tiny
15};
    \node (16) at (3.43443111294253,-5.22179882133085) {\tiny
16};
    \node (17) at (4.54608525983156,-4.28901023667958) {\tiny
17};
    \node (18) at (5.41265877365274,-3.12500000000000) {\tiny
18};
    \node (19) at (5.98743445197180,-1.79252020444432) {\tiny
19};
    \node (20) at (6.23942598919543,-0.363405180690477) {\tiny
20};
    \node (21) at (6.15504845632630,1.08530111041831) {\tiny
21};
    \node (22) at (5.73885066800171,2.47549853774473) {\tiny
22};
    \node (23) at (5.01326995471902,3.73224119814241) {\tiny
23};
    \node (24) at (4.01742256054087,4.78777776949361) {\tiny
24};
    \node (25) at (2.80499487625290,5.58520400202132) {\tiny
25};
    \node (26) at (1.44134919214025,6.08153044112390) {\tiny
26};
    \node (27) at (0.000000000000000,2.50000000000000) {\tiny
27};
    \node (28) at (-1.60696902421635,1.91511110779744) {\tiny
28};
    \node (29) at (-2.46201938253052,0.434120444167327) {\tiny
29};
    \node (30) at (-2.16506350946110,-1.25000000000000) {\tiny
30};
    \node (31) at (-0.855050358314169,-2.34923155196477) {\tiny
31};
    \node (32) at (0.855050358314165,-2.34923155196477) {\tiny
32};
    \node (33) at (2.16506350946110,-1.25000000000000) {\tiny
33};
    \node (34) at (2.46201938253052,0.434120444167323) {\tiny
34};
    \node (35) at (1.60696902421635,1.91511110779745) {\tiny
35};
    \path (0) edge (1);
    \path[ultra thick, red] (0) edge (26);
    \path (0) edge (27);
    \path (1) edge (2);
    \path (1) edge (23);
    \path (2) edge (3);
    \path (2) edge (7);
    \path (3) edge (4);
    \path (3) edge (28);
    \path (4) edge (5);
    \path (4) edge (26);
    \path (5) edge (6);
    \path (5) edge (10);
    \path (6) edge (7);
    \path (6) edge (29);
    \path (7) edge (8);
    \path (8) edge (9);
    \path (8) edge (13);
    \path (9) edge (10);
    \path (9) edge (30);
    \path (10) edge (11);
    \path[ultra thick, red] (11) edge (12);
    \path (11) edge (16);
    \path (12) edge (13);
    \path (12) edge (31);
    \path (13) edge (14);
    \path (14) edge (15);
    \path (14) edge (19);
    \path (15) edge (16);
    \path (15) edge (32);
    \path (16) edge (17);
    \path (17) edge (18);
    \path (17) edge (22);
    \path (18) edge (19);
    \path (18) edge (33);
    \path (19) edge (20);
    \path (20) edge (21);
    \path (20) edge (25);
    \path (21) edge (22);
    \path (21) edge (34);
    \path (22) edge (23);
    \path (23) edge (24);
    \path (24) edge (25);
    \path (24) edge (35);
    \path (25) edge (26);
    \path (27) edge (28);
    \path (27) edge (35);
    \path (28) edge (29);
    \path (29) edge (30);
    \path (30) edge (31);
    \path (31) edge (32);
    \path (32) edge (33);
    \path (33) edge (34);
    \path (34) edge (35);

    \foreach \x in {0,...,8} {
        \pgfmathsetmacro{\angle}{90+360*\x/9}
    	\node[draw=none] (b\x) at (\angle:7.2) {\scriptsize $b_{\x}$};
    }
    \foreach \x in {0,...,8} {
        \pgfmathsetmacro{\angle}{90+360*\x/9 + 360/27}
    	\node[draw=none] (b\x) at (\angle:7.2) {\scriptsize $d_{\x}$};
    }
    \foreach \x in {0,...,8} {
        \pgfmathsetmacro{\angle}{90+360*\x/9 - 360/27}
    	\node[draw=none] (b\x) at (\angle:7.2) {\scriptsize $c_{\x}$};
    }

    \foreach \x in {0,...,8} {
        \pgfmathsetmacro{\angle}{90+360*\x/9}
    	\node[draw=none] (b\x) at (\angle:1.6) {\scriptsize $a_{\x}$};
    }

\end{tikzpicture}

    \caption{The flower snark $J_{9}$. The suitable edges $b_0c_0$ and $b_4c_4$ are marked in bold red.}
    \label{fig:J9}
\end{figure}

\begin{lemma}[Steffen, Lemma 2.3 in~\cite{St01}]\label{lem:fixed}
The flower snarks $J_9$, $J_{11}$, and $J_{13}$ satisfy the conditions of Theorem~\ref{theorem:gz}.
\end{lemma}

\begin{proof} In \cite{St01}, in each of the graphs $J_9$, $J_{11}$, and $J_{13}$, the suitable edges were chosen to be $b_0c_0$ and $b_4c_4$. However, in~\cite{St01}, for various vertices $v$, the hamiltonian paths did not satisfy condition~(ii) from Theorem~\ref{theorem:gz}, as the paths used one of the edges $b_0c_0$ or $b_4c_4$. This was for instance the case for $v \in \{a_0,a_8,d_0\}$ in $J_9$, for $v \in \{a_0,a_{10},d_0\}$ in $J_{11}$ and for $v \in \{a_0,c_1,c_{12},d_0\}$ in $J_{13}$, see Claims 6, 7, and 8 in the Appendix of \cite{St01}.

We will now prove that $b_0c_0$ and $b_4c_4$ are indeed suitable edges for Theorem~\ref{theorem:gz} for $J_9$, $J_{11}$, and $J_{13}$. For $J_9$ the proof is given below (and partially in the Appendix), while the technical details of the proofs for $J_{11}$ and $J_{13}$ can be found in the Appendix. The mapping between the $a_i,b_i,c_i,d_i$ (used by Steffen) and the vertex numbers used in the proof is shown in Figures~\ref{fig:J9}--\ref{fig:J13}. We use numbers as labels in the proof to make it easier to read these graphs using a computer for verifying the results.

\medskip

\noindent
\textbf{Proof that $b_0c_0$ and $b_4c_4$ are suitable edges for $J_9$}.

Figure~\ref{fig:J9} shows the flower snark $J_9$. In $J_9$, $b_0c_0$ and $b_4c_4$ correspond to the edges $(0,26)$ and $(11, 12)$, respectively.

$(0,11), (0,12), (26,11)$ and $(26,12)$ are good in $J_9 - \{(0,26),(11,12)\}$ due to the following hamiltonian paths, respectively:

\begin{itemize}
{\small
\item 11, 10, 5, 6, 7, 8, 9, 30, 29, 28, 27, 35, 24, 23, 22, 17, 16, 15, 32, 31, 12, 13, 14, 19, 18, 33, 34, 21, 20, 25, 26, 4, 3, 2, 1, 0
\item 12, 13, 14, 15, 16, 11, 10, 5, 4, 26, 25, 20, 19, 18, 17, 22, 21, 34, 33, 32, 31, 30, 9, 8, 7, 6, 29, 28, 3, 2, 1, 23, 24, 35, 27, 0
\item 11, 10, 5, 4, 3, 2, 1, 0, 27, 28, 29, 6, 7, 8, 9, 30, 31, 12, 13, 14, 19, 18, 17, 16, 15, 32, 33, 34, 35, 24, 23, 22, 21, 20, 25, 26
\item 12, 13, 8, 7, 2, 3, 4, 5, 6, 29, 28, 27, 0, 1, 23, 22, 17, 18, 33, 32, 31, 30, 9, 10, 11, 16, 15, 14, 19, 20, 21, 34, 35, 24, 25, 26}
\end{itemize}


$((0,26),(11,12))$ is good in $J_9 - \{(0,26),(11,12)\}$ due to the following two disjoint paths with end-vertices $0$ and $26$, and $11$ and $12$, respectively, which together span $J_9$.

\begin{itemize}
{\small
\item 26, 25, 20, 19, 14, 13, 8, 7, 2, 1, 0
\item 12, 31, 32, 15, 16, 17, 18, 33, 34, 21, 22, 23, 24, 35, 27, 28, 3, 4, 5, 6, 29, 30, 9, 10, 11}
\end{itemize}

We showed by computer that at least one of $(0,26)$ or $(11,12)$ is good in $J_9 - \{v,(0,26),$ $(11,12)\}$ for every $v \in V(J_9)$. In each case we verified that the path found by the computer is indeed a valid hamiltonian path in the graph. Below we explicitly show this for $v = 0$. The hamiltonian paths for the other values of $v$ can be found in the Appendix.\\
{\small
\noindent $v = 0$: 12, 13, 14, 15, 32, 31, 30, 29, 6, 5, 10, 9, 8, 7, 2, 1, 23, 24, 25, 26, 4, 3, 28, 27, 35, 34, 33, 18, 19, 20, 21, 22, 17, 16, 11}
\end{proof}

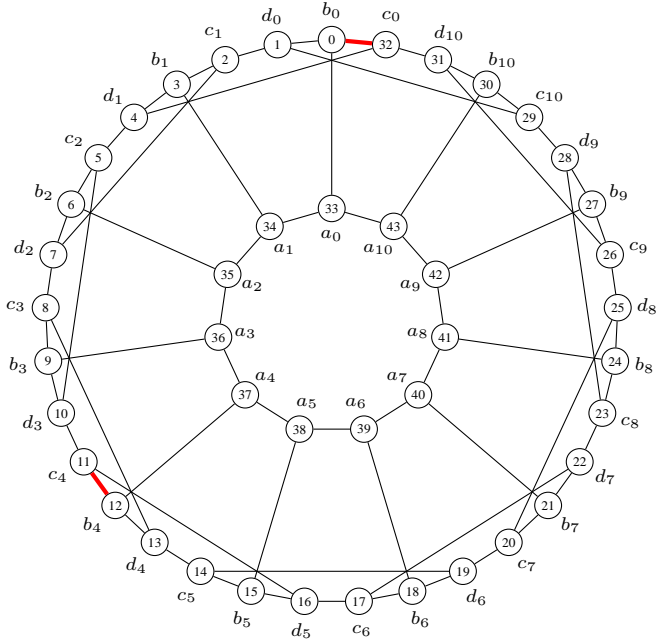
\begin{figure}[h!t]
	\centering

\begin{tikzpicture}[every node/.style={circle, draw, inner sep=1pt,
minimum size=10pt},every edge/.style={draw},scale=0.6]
    \node (0) at (0.000000000000000,6.25000000000000) {\tiny 0};
    \node (1) at (-1.18282027725256,6.13705435789192) {\tiny 1};
    \node (2) at (-2.32289034787705,5.80229958135045) {\tiny 2};
    \node (3) at (-3.37900510909748,5.25783458019488) {\tiny 3};
    \node (4) at (-4.31299382176320,4.52333773815669) {\tiny 4};
    \node (5) at (-5.09109970031460,3.62535568481999) {\tiny 5};
    \node (6) at (-5.68519997096574,2.59634383126179) {\tiny 6};
    \node (7) at (-6.07382230202214,1.47349334693392) {\tiny 7};
    \node (8) at (-6.24292086989380,0.297386973898389) {\tiny 8};
    \node (9) at (-6.18638401175583,-0.889467739208031) {\tiny 9};
    \node (10) at (-5.90625511696668,-2.04417477073388) {\tiny 10};
    \node (11) at (-5.41265877365274,-3.12500000000000) {\tiny 11};
    \node (12) at (-4.72343483971411,-4.09287958715803) {\tiny 12};
    \node (13) at (-3.86349366387878,-4.91283184214242) {\tiny 13};
    \node (14) at (-2.86391576079632,-5.55522155409327) {\tiny 14};
    \node (15) at (-1.76082848025894,-5.99683108509061) {\tiny 15};
    \node (16) at (-0.594100270651141,-6.22169951608178) {\tiny 16};
    \node (17) at (0.594100270651138,-6.22169951608178) {\tiny 17};
    \node (18) at (1.76082848025893,-5.99683108509061) {\tiny 18};
    \node (19) at (2.86391576079631,-5.55522155409327) {\tiny 19};
    \node (20) at (3.86349366387878,-4.91283184214242) {\tiny 20};
    \node (21) at (4.72343483971411,-4.09287958715804) {\tiny 21};
    \node (22) at (5.41265877365274,-3.12500000000000) {\tiny 22};
    \node (23) at (5.90625511696668,-2.04417477073389) {\tiny 23};
    \node (24) at (6.18638401175583,-0.889467739208033) {\tiny 24};
    \node (25) at (6.24292086989380,0.297386973898390) {\tiny 25};
    \node (26) at (6.07382230202214,1.47349334693392) {\tiny 26};
    \node (27) at (5.68519997096574,2.59634383126179) {\tiny 27};
    \node (28) at (5.09109970031460,3.62535568481998) {\tiny 28};
    \node (29) at (4.31299382176320,4.52333773815669) {\tiny 29};
    \node (30) at (3.37900510909749,5.25783458019488) {\tiny 30};
    \node (31) at (2.32289034787705,5.80229958135045) {\tiny 31};
    \node (32) at (1.18282027725257,6.13705435789192) {\tiny 32};
    \node (33) at (0.000000000000000,2.50000000000000) {\tiny 33};
    \node (34) at (-1.35160204363899,2.10313383207796) {\tiny 34};
    \node (35) at (-2.27407998838629,1.03853753250472) {\tiny 35};
    \node (36) at (-2.47455360470233,-0.355787095683213) {\tiny 36};
    \node (37) at (-1.88937393588565,-1.63715183486321) {\tiny 37};
    \node (38) at (-0.704331392103576,-2.39873243403624) {\tiny 38};
    \node (39) at (0.704331392103572,-2.39873243403624) {\tiny 39};
    \node (40) at (1.88937393588564,-1.63715183486322) {\tiny 40};
    \node (41) at (2.47455360470233,-0.355787095683217) {\tiny 41};
    \node (42) at (2.27407998838629,1.03853753250472) {\tiny 42};
    \node (43) at (1.35160204363899,2.10313383207795) {\tiny 43};
    \path (0) edge (1);
    \path[ultra thick, red] (0) edge (32);
    \path (0) edge (33);
    \path (1) edge (2);
    \path (1) edge (29);
    \path (2) edge (3);
    \path (2) edge (7);
    \path (3) edge (4);
    \path (3) edge (34);
    \path (4) edge (5);
    \path (4) edge (32);
    \path (5) edge (6);
    \path (5) edge (10);
    \path (6) edge (7);
    \path (6) edge (35);
    \path (7) edge (8);
    \path (8) edge (9);
    \path (8) edge (13);
    \path (9) edge (10);
    \path (9) edge (36);
    \path (10) edge (11);
    \path[ultra thick, red] (11) edge (12);
    \path (11) edge (16);
    \path (12) edge (13);
    \path (12) edge (37);
    \path (13) edge (14);
    \path (14) edge (15);
    \path (14) edge (19);
    \path (15) edge (16);
    \path (15) edge (38);
    \path (16) edge (17);
    \path (17) edge (18);
    \path (17) edge (22);
    \path (18) edge (19);
    \path (18) edge (39);
    \path (19) edge (20);
    \path (20) edge (21);
    \path (20) edge (25);
    \path (21) edge (22);
    \path (21) edge (40);
    \path (22) edge (23);
    \path (23) edge (24);
    \path (23) edge (28);
    \path (24) edge (25);
    \path (24) edge (41);
    \path (25) edge (26);
    \path (26) edge (27);
    \path (26) edge (31);
    \path (27) edge (28);
    \path (27) edge (42);
    \path (28) edge (29);
    \path (29) edge (30);
    \path (30) edge (31);
    \path (30) edge (43);
    \path (31) edge (32);
    \path (33) edge (34);
    \path (33) edge (43);
    \path (34) edge (35);
    \path (35) edge (36);
    \path (36) edge (37);
    \path (37) edge (38);
    \path (38) edge (39);
    \path (39) edge (40);
    \path (40) edge (41);
    \path (41) edge (42);
    \path (42) edge (43);

    \foreach \x in {0,...,10} {
        \pgfmathsetmacro{\angle}{90+360*\x/11}
    	\node[draw=none] (b\x) at (\angle:6.9) {\scriptsize $b_{\x}$};
    }
    \foreach \x in {0,...,10} {
        \pgfmathsetmacro{\angle}{90+360*\x/11 + 360/33}
    	\node[draw=none] (b\x) at (\angle:6.9) {\scriptsize $d_{\x}$};
    }
    \foreach \x in {0,...,10} {
        \pgfmathsetmacro{\angle}{90+360*\x/11 - 360/33}
    	\node[draw=none] (b\x) at (\angle:6.9) {\scriptsize $c_{\x}$};
    }

    \foreach \x in {0,...,10} {
        \pgfmathsetmacro{\angle}{90+360*\x/11}
    	\node[draw=none] (b\x) at (\angle:1.9) {\scriptsize $a_{\x}$};
    }

\end{tikzpicture}

    \caption{The flower snark $J_{11}$. The suitable edges $b_0c_0$ and $b_4c_4$ are marked in bold red.}
    \label{fig:J11}
\end{figure}

\begin{figure}[h!t]
	\centering
\begin{tikzpicture}[every node/.style={circle, draw, inner sep=1pt,minimum size=10pt},every edge/.style={draw},scale=0.7]
    \node (0) at (0.000000000000000,6.25000000000000) {\tiny 0};
    \node (1) at (-1.00257050536100,6.16906414148696) {\tiny 1};
    \node (2) at (-1.97917496125920,5.92835276216966) {\tiny 2};
    \node (3) at (-2.90451982527355,5.53410016033256) {\tiny 3};
    \node (4) at (-3.75463915148737,4.99651727127188) {\tiny 4};
    \node (5) at (-4.50751529589884,4.32952720943500) {\tiny 5};
    \node (6) at (-5.14364916183535,3.55040466706972) {\tiny 6};
    \node (7) at (-5.64656521631489,2.67932850876909) {\tiny 7};
    \node (8) at (-6.00323819769608,1.73885914947783) {\tiny 8};
    \node (9) at (-6.20443046311284,0.753354251595769) {\tiny 9};
    \node (10) at (-6.24493123857231,-0.251662125683843) {\tiny 10};
    \node (11) at (-6.12369157526417,-1.25016058610028) {\tiny 11};
    \node (12) at (-5.84385151678384,-2.21628054401585) {\tiny 12};
    \node (13) at (-5.41265877365274,-3.12500000000000) {\tiny 13};
    \node (14) at (-4.84128101142284,-3.95278359747111) {\tiny 14};
    \node (15) at (-4.14451661400497,-4.67819217606938) {\tiny 15};
    \node (16) at (-3.34041141329876,-5.28243803464872) {\tiny 16};
    \node (17) at (-2.44979131162547,-5.74987152286765) {\tiny 17};
    \node (18) at (-1.49572290179724,-6.06838635891283) {\tiny 18};
    \node (19) at (-0.502916054479536,-6.22973317583881) {\tiny 19};
    \node (20) at (0.502916054479533,-6.22973317583881) {\tiny 20};
    \node (21) at (1.49572290179723,-6.06838635891283) {\tiny 21};
    \node (22) at (2.44979131162547,-5.74987152286765) {\tiny 22};
    \node (23) at (3.34041141329876,-5.28243803464872) {\tiny 23};
    \node (24) at (4.14451661400497,-4.67819217606938) {\tiny 24};
    \node (25) at (4.84128101142284,-3.95278359747111) {\tiny 25};
    \node (26) at (5.41265877365274,-3.12500000000000) {\tiny 26};
    \node (27) at (5.84385151678384,-2.21628054401585) {\tiny 27};
    \node (28) at (6.12369157526417,-1.25016058610028) {\tiny 28};
    \node (29) at (6.24493123857231,-0.251662125683845) {\tiny 29};
    \node (30) at (6.20443046311284,0.753354251595764) {\tiny 30};
    \node (31) at (6.00323819769608,1.73885914947783) {\tiny 31};
    \node (32) at (5.64656521631489,2.67932850876909) {\tiny 32};
    \node (33) at (5.14364916183535,3.55040466706973) {\tiny 33};
    \node (34) at (4.50751529589884,4.32952720943500) {\tiny 34};
    \node (35) at (3.75463915148737,4.99651727127188) {\tiny 35};
    \node (36) at (2.90451982527355,5.53410016033256) {\tiny 36};
    \node (37) at (1.97917496125921,5.92835276216966) {\tiny 37};
    \node (38) at (1.00257050536100,6.16906414148696) {\tiny 38};
    \node (39) at (0.000000000000000,2.50000000000000) {\tiny 39};
    \node (40) at (-1.16180793010942,2.21364006413303) {\tiny 40};
    \node (41) at (-2.05745966473414,1.42016186682789) {\tiny 41};
    \node (42) at (-2.48177218524513,0.301341700638313) {\tiny 42};
    \node (43) at (-2.33754060671354,-0.886512217606337) {\tiny 43};
    \node (44) at (-1.65780664560199,-1.87127687042775) {\tiny 44};
    \node (45) at (-0.598289160718898,-2.42735454356513) {\tiny 45};
    \node (46) at (0.598289160718894,-2.42735454356513) {\tiny 46};
    \node (47) at (1.65780664560198,-1.87127687042776) {\tiny 47};
    \node (48) at (2.33754060671354,-0.886512217606342) {\tiny 48};
    \node (49) at (2.48177218524514,0.301341700638308) {\tiny 49};
    \node (50) at (2.05745966473414,1.42016186682789) {\tiny 50};
    \node (51) at (1.16180793010942,2.21364006413302) {\tiny 51};
    \path (0) edge (1);
    \path[ultra thick, red] (0) edge (38);
    \path (0) edge (39);
    \path (1) edge (2);
    \path (1) edge (35);
    \path (2) edge (3);
    \path (2) edge (7);
    \path (3) edge (4);
    \path (3) edge (40);
    \path (4) edge (5);
    \path (4) edge (38);
    \path (5) edge (6);
    \path (5) edge (10);
    \path (6) edge (7);
    \path (6) edge (41);
    \path (7) edge (8);
    \path (8) edge (9);
    \path (8) edge (13);
    \path (9) edge (10);
    \path (9) edge (42);
    \path (10) edge (11);
    \path[ultra thick, red] (11) edge (12);
    \path (11) edge (16);
    \path (12) edge (13);
    \path (12) edge (43);
    \path (13) edge (14);
    \path (14) edge (15);
    \path (14) edge (19);
    \path (15) edge (16);
    \path (15) edge (44);
    \path (16) edge (17);
    \path (17) edge (18);
    \path (17) edge (22);
    \path (18) edge (19);
    \path (18) edge (45);
    \path (19) edge (20);
    \path (20) edge (21);
    \path (20) edge (25);
    \path (21) edge (22);
    \path (21) edge (46);
    \path (22) edge (23);
    \path (23) edge (24);
    \path (23) edge (28);
    \path (24) edge (25);
    \path (24) edge (47);
    \path (25) edge (26);
    \path (26) edge (27);
    \path (26) edge (31);
    \path (27) edge (28);
    \path (27) edge (48);
    \path (28) edge (29);
    \path (29) edge (30);
    \path (29) edge (34);
    \path (30) edge (31);
    \path (30) edge (49);
    \path (31) edge (32);
    \path (32) edge (33);
    \path (32) edge (37);
    \path (33) edge (34);
    \path (33) edge (50);
    \path (34) edge (35);
    \path (35) edge (36);
    \path (36) edge (37);
    \path (36) edge (51);
    \path (37) edge (38);
    \path (39) edge (40);
    \path (39) edge (51);
    \path (40) edge (41);
    \path (41) edge (42);
    \path (42) edge (43);
    \path (43) edge (44);
    \path (44) edge (45);
    \path (45) edge (46);
    \path (46) edge (47);
    \path (47) edge (48);
    \path (48) edge (49);
    \path (49) edge (50);
    \path (50) edge (51);

    \foreach \x in {0,...,12} {
        \pgfmathsetmacro{\angle}{90+360*\x/13}
    	\node[draw=none] (b\x) at (\angle:6.9) {\scriptsize $b_{\x}$};
    }
    \foreach \x in {0,...,12} {
        \pgfmathsetmacro{\angle}{90+360*\x/13 + 360/39}
    	\node[draw=none] (b\x) at (\angle:6.9) {\scriptsize $d_{\x}$};
    }
    \foreach \x in {0,...,12} {
        \pgfmathsetmacro{\angle}{90+360*\x/13 - 360/38}
    	\node[draw=none] (b\x) at (\angle:6.9) {\scriptsize $c_{\x}$};
    }

    \foreach \x in {0,...,12} {
        \pgfmathsetmacro{\angle}{90+360*\x/13}
    	\node[draw=none] (b\x) at (\angle:1.9) {\scriptsize $a_{\x}$};
    }
\end{tikzpicture}
    \caption{The flower snark $J_{13}$. The suitable edges $b_0c_0$ and $b_4c_4$ are marked in bold red.}
    \label{fig:J13}
\end{figure}
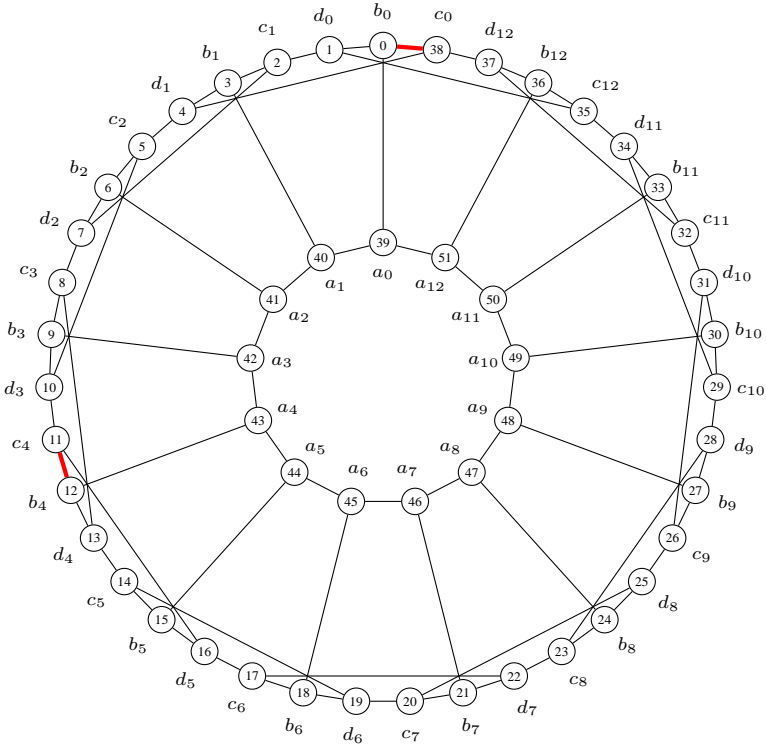






Since Steffen's statement of Lemma~\ref{lem:fixed} remains intact, the proof and statement of his main result, reproduced above as Theorem~\ref{Steffen}, are correct as given in~\cite{St01}. Even though we prove a stronger version of Steffen's theorem in the next section, we think it is important to fix the proof of Lemma~\ref{lem:fixed} as there may be others who rely on this lemma, or might want to rely on it in the future.

\subsection{Orders of hypohamiltonian snarks}
\label{subsect:orders}

We shall now prove a strengthening of Steffen's theorem, which in a sense is strongest possible since we will determine all orders for which hypohamiltonian snarks exist. We emphasise that our proof's mechanism contains significantly fewer ``moving parts'' than M\'{a}\v{c}ajov\'{a} and \v{S}koviera's~\cite{MS06}, and, as mentioned in the introduction, our theorem also strengthens their result. We do need the following two easily verifiable lemmas.

\begin{lemma}\label{lem:B}
The second Blanu\v{s}a snark $B_2$ shown in Figure~\ref{fig:B2} has a pair of suitable edges.
\end{lemma}

\begin{proof}
Figure~\ref{fig:B2} shows the second Blanu\v{s}a snark $B_2$. By computer we determined that $B_2$ has exactly three pairs of suitable edges: $((6,8),(10,16))$, $((3,9),(12,17))$ and $((4,7),$ $(13,15))$. We will now prove by hand that $((6,8),(10,16))$ is a suitable edge pair.

$(6,10), (6,16), (8,10)$ and $(8,16)$ are good in $B_2 - \{(6,8),(10,16)\}$ due to the following hamiltonian paths, respectively:

\begin{itemize}
{\small
\item 10, 11, 12, 17, 16, 15, 13, 14, 0, 1, 5, 4, 7, 8, 9, 3, 2, 6
\item 16, 15, 9, 8, 7, 17, 12, 13, 14, 10, 11, 1, 0, 2, 3, 4, 5, 6
\item 10, 11, 1, 0, 14, 13, 12, 17, 16, 15, 9, 3, 2, 6, 5, 4, 7, 8
\item 16, 15, 9, 3, 4, 5, 6, 2, 0, 1, 11, 10, 14, 13, 12, 17, 7, 8}
\end{itemize}

$((6,8),(10,16))$ is good in $B_2 - \{(6,8),(10,16)\}$ due to the following two disjoint paths with end-vertices $6$ and $8$, and $10$ and $16$, respectively, which together span $B_2$.

\begin{itemize}
{\small
\item 8, 7, 4, 5, 6
\item 16, 17, 12, 11, 1, 0, 2, 3, 9, 15, 13, 14, 10}
\end{itemize}

We now prove that at least one of $(6,8)$ or $(10,16)$ is good in $B_2 - \{v,(6,8),(10,16)\}$ for every $v \in V(B_2)$. By symmetry we only need to prove this for $v=0,2,4,6,7,8$.
\begin{itemize}
{\footnotesize
\item $v = 0$: 8, 7, 4, 5, 1, 11, 10, 14, 13, 12, 17, 16, 15, 9, 3, 2, 6
\item $v = 2$: 8, 9, 3, 4, 7, 17, 16, 15, 13, 12, 11, 10, 14, 0, 1, 5, 6
\item $v = 4$: 16, 15, 13, 14, 0, 1, 5, 6, 2, 3, 9, 8, 7, 17, 12, 11, 10
\item $v = 6$: 16, 15, 9, 8, 7, 17, 12, 13, 14, 0, 2, 3, 4, 5, 1, 11, 10
\item $v = 7$: 8, 9, 15, 16, 17, 12, 13, 14, 10, 11, 1, 0, 2, 3, 4, 5, 6
\item $v = 8$: 16, 15, 9, 3, 2, 6, 5, 4, 7, 17, 12, 13, 14, 0, 1, 11, 10
}
\end{itemize}
\end{proof}

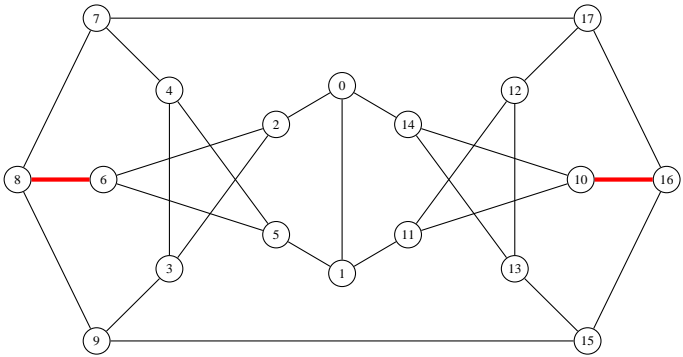
\begin{figure}[h!t]
	\centering       	
\begin{tikzpicture}[every node/.style={circle, draw, inner sep=1pt,
minimum size=10pt},every edge/.style={draw},scale=0.5]
    \node (0) at (0,2.50000000000000) {\tiny 0};
    \node (1) at (0,-2.50000000000000) {\tiny 1};
    \node (2) at (-1.72745751406263,1.46946313073118) {\tiny 2};
    \node (3) at (-4.52254248593737,-2.37764129073788) {\tiny 3};
    \node (4) at (-4.52254248593737,2.37764129073788) {\tiny 4};
    \node (5) at (-1.72745751406263,-1.46946313073118) {\tiny 5};
    \node (6) at (-6.25000000000000,0) {\tiny 6};
    \node (7) at (-6.42919391022303,4.30007315357417) {\tiny 7};
    \node (8) at (-8.50000000000000,0) {\tiny 8};
    \node (9) at (-6.42919391022304,-4.30007315357416) {\tiny 9};
    \node (10) at (6.25000000000000,0.000000000000000) {\tiny 10};
    \node (11) at (1.72745751406263,-1.46946313073118) {\tiny 11};
    \node (12) at (4.52254248593737,2.37764129073788) {\tiny 12};
    \node (13) at (4.52254248593737,-2.37764129073788) {\tiny 13};
    \node (14) at (1.72745751406263,1.46946313073118) {\tiny 14};
    \node (15) at (6.42919391022303,-4.30007315357416) {\tiny 15};
    \node (16) at (8.50000000000000,0.000000000000000) {\tiny 16};
    \node (17) at (6.42919391022303,4.30007315357416) {\tiny 17};
    \path (0) edge (1);
    \path (0) edge (2);
    \path (0) edge (14);
    \path (1) edge (5);
    \path (1) edge (11);
    \path (2) edge (3);
    \path (2) edge (6);
    \path (3) edge (4);
    \path (3) edge (9);
    \path (4) edge (5);
    \path (4) edge (7);
    \path (5) edge (6);
    \path[ultra thick, red] (6) edge (8);
    \path (7) edge (8);
    \path (7) edge (17);
    \path (8) edge (9);
    \path (9) edge (15);
    \path (10) edge (11);
    \path (10) edge (14);
    \path[ultra thick, red] (10) edge (16);
    \path (11) edge (12);
    \path (12) edge (13);
    \path (12) edge (17);
    \path (13) edge (14);
    \path (13) edge (15);
    \path (15) edge (16);
    \path (16) edge (17);
\end{tikzpicture}

    \caption{The second Blanu\v{s}a snark. It has 18 vertices. The suitable edges $(6,8)$ and $(10,16)$ are marked in bold red.}
    \label{fig:B2}
\end{figure}

\newpage 

\begin{lemma}\label{lem:22}
The first Loupekine snark $L_1$ shown in Figure~\ref{fig:L1} has a pair of suitable edges.
\end{lemma}

\begin{proof}

Figure~\ref{fig:L1} shows the first Loupekine snark~$L_1$. By computer we determined that $L_1$ has exactly six pairs of suitable edges: $((0,1),(17,20))$, $((0,2),(8,17))$, $((1,5),(14,20))$, $((2,3),(8,10))$, $((3,4),(10,12))$ and $((4,5),(12,14))$. We will now prove by hand that $((3,4),(10,12))$ is a suitable edge pair.

$(3,10), (3,12), (4,10)$ and $(4,12)$ are good in $L_1 - \{(3,4),(10,12)\}$ due to the following hamiltonian paths, respectively:

\begin{itemize}
{\small
\item 10, 7, 9, 6, 8, 17, 19, 21, 16, 13, 11, 0, 1, 5, 4, 18, 20, 14, 12, 15, 2, 3
\item 12, 14, 20, 18, 4, 5, 7, 10, 8, 17, 19, 21, 16, 9, 6, 1, 0, 11, 13, 15, 2, 3
\item 10, 7, 5, 1, 0, 11, 13, 16, 9, 6, 8, 17, 20, 14, 12, 15, 2, 3, 19, 21, 18, 4
\item 12, 14, 11, 0, 1, 6, 9, 16, 13, 15, 2, 3, 19, 21, 18, 20, 17, 8, 10, 7, 5, 4
}
\end{itemize}

$((3,4),(10,12))$ is good in $L_1 - \{(3,4),(10,12)\}$ due to the following two disjoint paths with end-vertices $3$ and $4$, and $10$ and $12$, respectively, which together span $L_1$.

\begin{itemize}
{\small
\item 4, 5, 1, 6, 8, 17, 20, 18, 21, 19, 3
\item 12, 14, 11, 0, 2, 15, 13, 16, 9, 7, 10}
\end{itemize}

We now prove that at least one of $(3,4)$ or $(10,12)$ is good in $L_1 - \{v,(3,4),(10,12)\}$ for every $v \in V(L_1)$. By symmetry we only need to prove this for $v=1,4,5,6,7,8,9,10,$ $16,17,18,21$.

\begin{itemize}
{\footnotesize
\item $v = 1$: 12, 15, 13, 16, 9, 6, 8, 17, 20, 14, 11, 0, 2, 3, 19, 21, 18, 4, 5, 7, 10
\item $v = 4$: 12, 14, 20, 18, 21, 16, 9, 6, 8, 17, 19, 3, 2, 15, 13, 11, 0, 1, 5, 7, 10
\item $v = 5$: 4, 18, 20, 17, 8, 10, 7, 9, 6, 1, 0, 2, 15, 12, 14, 11, 13, 16, 21, 19, 3
\item $v = 6$: 4, 5, 1, 0, 2, 15, 12, 14, 11, 13, 16, 9, 7, 10, 8, 17, 20, 18, 21, 19, 3
\item $v = 7$: 12, 14, 20, 17, 19, 3, 2, 15, 13, 11, 0, 1, 5, 4, 18, 21, 16, 9, 6, 8, 10
\item $v = 8$: 12, 14, 20, 17, 19, 3, 2, 15, 13, 11, 0, 1, 6, 9, 16, 21, 18, 4, 5, 7, 10
\item $v = 9$: 4, 5, 7, 10, 8, 6, 1, 0, 2, 15, 12, 14, 11, 13, 16, 21, 18, 20, 17, 19, 3
\item $v = 10$: 4, 5, 7, 9, 16, 13, 11, 0, 1, 6, 8, 17, 19, 21, 18, 20, 14, 12, 15, 2, 3
\item $v = 16$: 4, 5, 1, 6, 9, 7, 10, 8, 17, 19, 21, 18, 20, 14, 12, 15, 13, 11, 0, 2, 3
\item $v = 17$: 4, 5, 1, 6, 8, 10, 7, 9, 16, 13, 11, 0, 2, 15, 12, 14, 20, 18, 21, 19, 3
\item $v = 18$: 4, 5, 1, 6, 8, 10, 7, 9, 16, 21, 19, 17, 20, 14, 12, 15, 13, 11, 0, 2, 3
\item $v = 21$: 4, 18, 20, 14, 12, 15, 2, 0, 11, 13, 16, 9, 6, 1, 5, 7, 10, 8, 17, 19, 3
}
\end{itemize}
\end{proof}

\begin{figure}[h!t]
	\centering       	

\begin{tikzpicture}[every node/.style={circle, draw, inner sep=1pt,
minimum size=10pt},every edge/.style={draw},scale=0.5]
    \node (0) at (13.1500000000000,14.9000000000000) {\tiny 0};
    \node (1) at (4.65000000000000,14.9500000000000) {\tiny 1};
    \node (2) at (15.7500000000000,8.36111111111111) {\tiny 2};
    \node (3) at (13.0500000000000,1.70000000000000) {\tiny 3};
    \node (4) at (4.54069767441861,1.84722222222222) {\tiny 4};
    \node (5) at (2.05000000000000,8.35000000000000) {\tiny 5};
    \node (6) at (4.72674418604651,11.9097222222222) {\tiny 6};
    \node (7) at (4.30813953488372,9.91666666666667) {\tiny 7};
    \node (8) at (5.44978858350951,8.96955719557196) {\tiny 8};
    \node (9) at (6.68023255813953,10.0138888888889) {\tiny 9};
    \node (10) at (6.44767441860465,11.8125000000000) {\tiny 10};
    \node (11) at (12.5956659619450,11.9830873308733) {\tiny 11};
    \node (12) at (10.7732558139535,11.7638888888889) {\tiny 12};
    \node (13) at (10.6802325581395,9.62500000000000) {\tiny 13};
    \node (14) at (12.5872093023256,8.99305555555556) {\tiny 14};
    \node (15) at (13.5681818181818,10.2754202542025) {\tiny 15};
    \node (16) at (8.63372093023256,8.40972222222222) {\tiny 16};
    \node (17) at (7.19186046511628,5.49305555555556) {\tiny 17};
    \node (18) at (7.47093023255814,3.69444444444444) {\tiny 18};
    \node (19) at (9.61046511627907,3.64583333333333) {\tiny 19};
    \node (20) at (9.88953488372093,5.20138888888889) {\tiny 20};
    \node (21) at (8.58720930232558,6.22222222222222) {\tiny 21};
    \path (0) edge (1);
    \path (0) edge (2);
    \path (0) edge (11);
    \path (1) edge (5);
    \path (1) edge (6);
    \path (2) edge (3);
    \path (2) edge (15);
    \path[ultra thick, red] (3) edge (4);
    \path (3) edge (19);
    \path (4) edge (5);
    \path (4) edge (18);
    \path (5) edge (7);
    \path (6) edge (8);
    \path (6) edge (9);
    \path (7) edge (9);
    \path (7) edge (10);
    \path (8) edge (10);
    \path (8) edge (17);
    \path (9) edge (16);
    \path[ultra thick, red] (10) edge (12);
    \path (11) edge (13);
    \path (11) edge (14);
    \path (12) edge (14);
    \path (12) edge (15);
    \path (13) edge (15);
    \path (13) edge (16);
    \path (14) edge (20);
    \path (16) edge (21);
    \path (17) edge (19);
    \path (17) edge (20);
    \path (18) edge (20);
    \path (18) edge (21);
    \path (19) edge (21);
\end{tikzpicture}

    \caption{The first Loupekine snark $L_1$. It has 22 vertices. The suitable edges $(3,4)$ and $(10,12)$ are marked in bold red.}
    \label{fig:L1}
\end{figure}
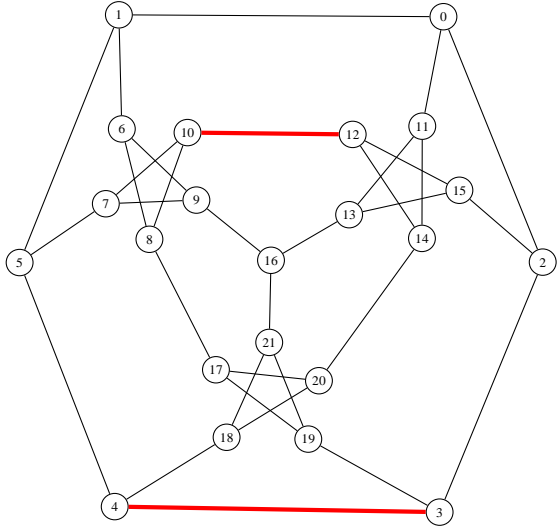

\newpage

The following is the strongest form of Steffen's Theorem~\ref{Steffen}.

\begin{theorem}\label{thm:hyposnarks}
There exists a hypohamiltonian snark of order $n$ iff $n \in \{ 10, 18, 20, 22 \}$ or $n$ is even and $n \ge 26$.
\end{theorem}

\begin{proof}
For $n = 10$, it is well-known that the Petersen graph is hypohamiltonian and it is also well-known that no snarks exist of order~12, 14 or 16.

In Lemma~\ref{lem:B} we showed that the second Blanu\v{s}a snark $B_2$ (which has order~18) contains a pair of suitable edges. In~\cite{BGHM13} it was proven that hypohamiltonian snarks exists for all even orders from 18 to 36 with the exception of 24 (see Table~\ref{table:numhypo}). Let $S_n$ denote a hypohamiltonian snark of order~$n$. Using Theorem~\ref{theorem:gz}, we form the dot product $B_2 \cdot S_n$ for $n \in \{ 18, 20, 22, 26, 28, 30, 32 \}$ and obtain hypohamiltonian snarks of all even orders between 34 and 48 with the exception of 40 (recall that the dot product of two snarks is a snark).

By Lemma~\ref{lem:22} we know that the first Loupekine snark $L_1$ (which has order~22) contains a pair of suitable edges. Applying Theorem~\ref{theorem:gz} to this snark and a 22-vertex hypohamiltonian snark, we obtain a hypohamiltonian snark of order~40.

We form the dot product $B_2 \cdot S_n$ for all even $n \in \{ 34, ..., 48 \}$ and obtain hypohamiltonian snarks of all even orders from 50 to 64. This may now be iterated ad infinitum, and the proof is complete.
\end{proof}

In~\cite{MS06} M\'{a}\v{c}ajov\'{a} and \v{S}koviera proved the following theorem (which fully settles Problem~\ref{problem:NS}).

\begin{theorem}[M\'{a}\v{c}ajov\'{a} and \v{S}koviera, Theorems A and B in~\cite{MS06}]\label{thm:ms06}
There exists an irreducible snark of order $n$ iff $n \in \{ 10, 18, 20, 22 \}$ or $n$ is even and $n \ge 26$.
\end{theorem}

Nedela and \v{S}koviera~\cite{NS96} proved that a snark is irreducible if and only if it is bicritical, and Steffen~\cite{St98} showed that every hypohamiltonian snark is bicritical---while the converse is not true, as will be shown in Table~\ref{table:numhypo}.

In~\cite{carneiro2015faster} Carneiro, da Silva, and McKay determined all 4-vertex-critical snarks up to 36 vertices and \v{S}koviera~\cite{skoviera-pers} showed that a snark is 4-vertex-critical if and only if it is irreducible. 
(A graph $G$ without $k$-flow is $k$-vertex-critical if, for every pair of vertices $(u,v)$ of $G$, identifying $u$ and $v$ yields a graph that has a $k$-flow (see~\cite{carneiro2015faster} for more details)).

The number of hypohamiltonian snarks on $n \le 36$ vertices was determined in~\cite{BGHM13} and can be found in Table~\ref{table:numhypo} together with the number of irreducible snarks from~\cite{carneiro2015faster}. These graphs can also be downloaded from the \textit{House of Graphs}~\cite{hog} at \\ \url{http://hog.grinvin.org/Snarks}.

As can be seen from Table~\ref{table:numhypo}, there are a significant number of irreducible snarks which are not hypohamiltonian. The smallest such snarks have order 26. So Theorem~\ref{thm:hyposnarks} implies Theorem~\ref{thm:ms06}, while the inverse implication is not true.

\begin{table}[ht!]
\centering
		\begin{tabular}{| c | r | r | r | r |}
		\hline
		Order & irreducible & hypo. & hypo. and $\lambda_c = 4$ & hypo. and $\lambda_c \geq 5$\\
		\hline
10  &   1  &   1  & 0 &  1\\
18  &   2  &   2  & 2 &  0\\
20  &   1  &   1  & 0 &  1\\
22  &   2  &   2  & 0 &  2\\
24  &   0  &   0  & 0 &  0\\
26  &   111  &   95  & 87 &  8\\
28  &   33  &   31  &  30 &   1\\
30  &   115  &   104  & 93 &  11\\
32  &   13  &   13  & 0 &   13\\
34  &   40 328  &   31 198  & 29 701 &  1 497\\
36  &   13 720  &   10 838  &  10 374 & 464\\
38  &   ?  &   ?  &  $\ge $ 51 431 & ?\\
40  &   ?  &   ?  &  $\ge $ 8 820 & ?\\
42  &   ?  &   ?  &  $\ge $ 20 575 458 & ?\\
44  &   ?  &   ?  &  $\ge $ 8 242 146 & ?\\
		\hline
		\end{tabular}
		\caption{Number of irreducible and hypohamiltonian snarks (see~\cite[Table 1]{carneiro2015faster} and~\cite[Table 2]{BGHM13}). $\lambda_c$~stands for cyclic edge-connectivity. The counts of cases indicated with a '$\ge$' are possibly incomplete; all other cases are complete.}
		\label{table:numhypo}
\end{table}

The hypohamiltonian snarks on $n \ge 34$ vertices constructed by the dot product in the proof of Theorem~\ref{thm:hyposnarks} clearly all have cyclic edge-connectivity~4. By combining this with Table~\ref{table:numhypo} we obtain:

\begin{corollary}\label{thm:hyposnarks-cyc4}
There exists a hypohamiltonian snark of order $n$ and cyclic edge-connectivity~$4$ iff $n \in \{ 18, 26, 28, 30 \}$ or $n$ is even and $n \ge 34$.
\end{corollary}

As already mentioned, every hypohamiltonian snark is irreducible, thus Corollary~\ref{thm:hyposnarks-cyc4} implies~\cite[Theorem~A]{MS06}. For higher cyclic edge-connectivity, the following is known.

\begin{theorem}[M\'a\v{c}ajov\'a and \v{S}koviera~\cite{MS07}]\label{thm:hyposnarks-cyc56}
There exists a hypohamiltonian snark of order~$n$ and cyclic connectivity~$5$ for each even $n \ge 140$, and there exists a hypohamiltonian snark of order~$n$ and cyclic connectivity~$6$ for each even $n \ge 166$.
\end{theorem}

If we relax the requirements from hypohamiltonicity to irreducibility, more is known:

\begin{theorem}[M\'a\v{c}ajov\'a and \v{S}koviera~\cite{MS06}]\label{thm:irredsnarks-cyc56}
There exists an irreducible snark of order~$n$ and cyclic connectivity~$5$ iff $n \in \{ 10, 20, 22, 26 \}$ or $n$ is even and $n \ge 30$, and there exists an irreducible snark of order~$n$ and cyclic connectivity~$6$ for each $n \equiv 4 \ (\textup{mod} \ 8)$ with $n \ge 28$, and for each even $n \ge 210$.
\end{theorem}

Note that as every hypohamiltonian graph is irreducible, Theorem~\ref{thm:hyposnarks-cyc56} also implies that $n \ge 210$ can be improved to $n \ge 166$ in Theorem~\ref{thm:irredsnarks-cyc56}.

The smallest hypohamiltonian snark of cyclic edge-connectivity~5 has order~10 and is the Petersen graph, and the second-smallest such graph has order~20. The flower snark $J_7$ of order~28 is the smallest cyclically 6-edge-connected hypohamiltonian snark.  We conclude this section with the following two problems motivated by Theorem~\ref{thm:irredsnarks-cyc56} and results of Kochol~\cite{Ko96-2,Ko04}.

\begin{problem}[M\'a\v{c}ajov\'a and \v{S}koviera~\cite{MS06}]
Construct a cyclically $6$-edge-connected snark (irreducible or not) of order smaller than~$118$ and different from any of Isaacs' snarks.
\end{problem}

\begin{problem}
Determine all orders for which cyclically $6$-edge-connected snarks exist.
\end{problem}

\section{On a Conjecture of Steffen on Hypohamiltonian Snarks}
\label{sect:conjecture_steffen}

Consider a cubic graph $G$. We denote with $\mu_k(G)$ the minimum number of edges not contained in the union of $k$ 1-factors of $G$, for every possible combination of $k$ 1-factors. If $\mu_3(G) = 0$, then $G$ is 3-edge-colourable. In~\cite{steffen20151}, Steffen made the following conjecture on hypohamiltonian snarks.

\begin{conjecture}[Steffen, Conjecture 4.1 in~\cite{steffen20151}]\label{conj:steffen}
If $G$ is a hypohamiltonian snark, then $\mu_3(G) = 3$.
\end{conjecture}

If true, this conjecture would have significant consequences, e.g.\ by Theorem~2.14 from~\cite{steffen20151}, it would imply that every hypohamiltonian snark has a Berge-cover (a bridgeless cubic graph $G$ has a \emph{Berge-cover} if $\mu_5(G) = 0$).

We wrote a computer program for computing $\mu_3(G)$ and tested Conjecture~\ref{conj:steffen} on the complete lists of hypohamiltonian snarks up to 36~vertices. This leads to the following observation.

\begin{observation}
There are no counterexamples to Conjecture~\ref{conj:steffen} among the hypohamiltonian snarks with at most 36 vertices.
\end{observation}

The authors of~\cite{BGHM13} noted a huge increase (from 13 to 31~198) in the number of hypohamiltonian snarks from order~32 to 34, see Table~\ref{table:numhypo}. Using a computer, we were able to determine that 29~365 out of the 29~701 hypohamiltonian snarks with cyclic edge-connectivity~4 on 34~vertices can be obtained by performing a dot product on a hypohamiltonian snark on 26~vertices and the Petersen graph. We also determined that the remaining hypohamiltonian snarks with cyclic edge-connectivity~4 on 34~vertices are obtained by performing a dot product on the Blanu\v{s}a snarks. Intriguingly, our computations show that some hypohamiltonian snarks can be obtained by performing a dot product on a hypohamiltonian snark on 26~vertices and the Petersen graph, as well as by performing a dot product on the Blanu\v{s}a snarks.

There is also a (slightly less dramatic) increase in the cyclically 5-edge-connected case---these are obviously not dot products---and we believe it to be due to more general graph products, for instance ``superposition'' introduced by Kochol~\cite{Ko96}. It would be interesting to further explore these transformations in order to fully understand these sudden increases and decreases in numbers.

Using a computer, we determined that all hypohamiltonian snarks with cyclic edge-connectivity~4 up to 36~vertices can be obtained by performing a dot product on two hypohamiltonian snarks. This leads us to pose the following question.

\begin{problem}
Is every hypohamiltonian snark with cyclic edge-connectivity 4 a dot product of two hypohamiltonian snarks?
\end{problem}

Recall that in Theorem~\ref{theorem:gz} the graphs $G$ and $H$ are hypohamiltonian, but the theorem is \emph{not} ``if and only if'', since although the Petersen graph does not contain a pair of suitable edges, the Blanu\v{s}a snarks (which are dot products of two Petersen graphs) are hypohamiltonian. We believe the answer to this problem to be ``no'' due to the following observation. In order to cover all cases, we would need to add to condition~(ii) of Theorem~\ref{theorem:gz} the possibility of $((a,b),(c,d))$ being good in $G - \{ v, ab, cd \}$. However, we would then need to require from $H$ that it contains a 2-factor containing exactly two (necessarily odd) cycles. Although we were unable to find a counter-example, we believe that there exist hypohamiltonian snarks which do not possess such a 2-factor. 

We also determined all hypohamiltonian snarks up to 44 vertices which can be obtained by performing a dot product on two hypohamiltonian snarks. The counts of these snarks can be found in the fourth column of Table~\ref{table:numhypo}. We also verified Conjecture~\ref{conj:steffen} on these snarks.

\begin{observation}
There are no counterexamples to Conjecture~\ref{conj:steffen} among the hypohamiltonian snarks with at most 44 vertices which are a dot product of two hypohamiltonian snarks.
\end{observation}

\section*{Acknowledgements}
We thank Nico Van Cleemput for providing us with a script which greatly enhanced the quality of our figures. We would also like to thank Martin \v{S}koviera for informing us about the equivalence of irreducible and vertex-critical graphs. Finally, we also wish to thank Eckhard Steffen for useful suggestions.



\bibliographystyle{amcjoucc}
\bibliography{references}

\newpage

\section*{Appendix}

Below we give the technical details which were left out in the proof of Lemma~\ref{lem:fixed}.

\bigskip

\noindent
\textbf{Proof that $b_0c_0$ and $b_4c_4$ are suitable edges for $J_{9}$ (continued)}

We will now prove that at least one of $(0,26)$ or $(11,12)$ is good in $J_9 - \{v,(0,26),$ $(11,12)\}$ for every $v \in V(J_9)$ except for $v = 0$, which we have already shown above in the proof of Lemma~\ref{lem:fixed}.

\begin{itemize}
{\footnotesize
\item $v = 1$: 26, 4, 5, 6, 7, 2, 3, 28, 29, 30, 31, 12, 13, 8, 9, 10, 11, 16, 17, 18, 33, 32, 15, 14, 19, 20, 25, 24, 23, 22, 21, 34, 35, 27, 0
\item $v = 2$: 12, 31, 32, 15, 14, 13, 8, 7, 6, 5, 10, 9, 30, 29, 28, 3, 4, 26, 25, 24, 23, 1, 0, 27, 35, 34, 33, 18, 19, 20, 21, 22, 17, 16, 11
\item $v = 3$: 26, 4, 5, 6, 7, 2, 1, 23, 22, 21, 20, 25, 24, 35, 34, 33, 32, 15, 14, 19, 18, 17, 16, 11, 10, 9, 8, 13, 12, 31, 30, 29, 28, 27, 0
\item $v = 4$: 26, 25, 20, 19, 18, 17, 22, 21, 34, 33, 32, 31, 12, 13, 14, 15, 16, 11, 10, 5, 6, 7, 8, 9, 30, 29, 28, 3, 2, 1, 23, 24, 35, 27, 0
\item $v = 5$: 26, 4, 3, 2, 1, 23, 22, 21, 20, 25, 24, 35, 34, 33, 32, 15, 14, 19, 18, 17, 16, 11, 10, 9, 30, 31, 12, 13, 8, 7, 6, 29, 28, 27, 0
\item $v = 6$: 26, 25, 20, 19, 18, 17, 22, 21, 34, 33, 32, 31, 12, 13, 14, 15, 16, 11, 10, 5, 4, 3, 2, 7, 8, 9, 30, 29, 28, 27, 35, 24, 23, 1, 0
\item $v = 7$: 26, 4, 5, 6, 29, 30, 31, 12, 13, 8, 9, 10, 11, 16, 17, 18, 19, 14, 15, 32, 33, 34, 35, 24, 25, 20, 21, 22, 23, 1, 2, 3, 28, 27, 0
\item $v = 8$: 26, 25, 20, 19, 18, 17, 22, 21, 34, 33, 32, 31, 12, 13, 14, 15, 16, 11, 10, 9, 30, 29, 28, 3, 4, 5, 6, 7, 2, 1, 23, 24, 35, 27, 0
\item $v = 9$: 26, 4, 3, 2, 1, 23, 22, 21, 20, 25, 24, 35, 34, 33, 32, 15, 14, 19, 18, 17, 16, 11, 10, 5, 6, 7, 8, 13, 12, 31, 30, 29, 28, 27, 0
\item $v = 10$: 12, 13, 14, 15, 32, 31, 30, 9, 8, 7, 2, 3, 28, 29, 6, 5, 4, 26, 25, 24, 23, 1, 0, 27, 35, 34, 33, 18, 19, 20, 21, 22, 17, 16, 11
\item $v = 11$: 26, 4, 3, 2, 1, 23, 22, 21, 34, 35, 24, 25, 20, 19, 14, 15, 16, 17, 18, 33, 32, 31, 12, 13, 8, 7, 6, 5, 10, 9, 30, 29, 28, 27, 0
\item $v = 12$: 26, 25, 20, 19, 18, 17, 22, 21, 34, 33, 32, 31, 30, 9, 8, 13, 14, 15, 16, 11, 10, 5, 4, 3, 2, 7, 6, 29, 28, 27, 35, 24, 23, 1, 0
\item $v = 13$: 12, 31, 30, 29, 6, 5, 10, 9, 8, 7, 2, 1, 0, 27, 28, 3, 4, 26, 25, 20, 21, 22, 23, 24, 35, 34, 33, 32, 15, 14, 19, 18, 17, 16, 11
\item $v = 14$: 26, 25, 20, 19, 18, 17, 22, 21, 34, 33, 32, 15, 16, 11, 10, 9, 8, 13, 12, 31, 30, 29, 28, 3, 4, 5, 6, 7, 2, 1, 23, 24, 35, 27, 0
\item $v = 15$: 26, 4, 3, 2, 1, 23, 22, 21, 20, 25, 24, 35, 34, 33, 32, 31, 12, 13, 14, 19, 18, 17, 16, 11, 10, 5, 6, 7, 8, 9, 30, 29, 28, 27, 0
\item $v = 16$: 12, 13, 8, 7, 2, 3, 28, 29, 6, 5, 4, 26, 25, 24, 23, 1, 0, 27, 35, 34, 33, 18, 17, 22, 21, 20, 19, 14, 15, 32, 31, 30, 9, 10, 11
\item $v = 17$: 26, 4, 3, 2, 1, 23, 22, 21, 20, 25, 24, 35, 34, 33, 18, 19, 14, 13, 12, 31, 32, 15, 16, 11, 10, 5, 6, 7, 8, 9, 30, 29, 28, 27, 0
\item $v = 18$: 26, 25, 20, 19, 14, 15, 32, 33, 34, 21, 22, 17, 16, 11, 10, 9, 8, 13, 12, 31, 30, 29, 28, 3, 4, 5, 6, 7, 2, 1, 23, 24, 35, 27, 0
\item $v = 19$: 26, 4, 3, 28, 27, 35, 34, 21, 20, 25, 24, 23, 22, 17, 18, 33, 32, 31, 12, 13, 14, 15, 16, 11, 10, 5, 6, 29, 30, 9, 8, 7, 2, 1, 0
\item $v = 20$: 26, 25, 24, 35, 27, 28, 29, 6, 5, 4, 3, 2, 7, 8, 13, 12, 31, 30, 9, 10, 11, 16, 17, 18, 19, 14, 15, 32, 33, 34, 21, 22, 23, 1, 0
\item $v = 21$: 26, 4, 3, 2, 1, 23, 22, 17, 18, 19, 20, 25, 24, 35, 34, 33, 32, 31, 12, 13, 14, 15, 16, 11, 10, 5, 6, 7, 8, 9, 30, 29, 28, 27, 0
\item $v = 22$: 26, 25, 20, 21, 34, 33, 32, 15, 14, 19, 18, 17, 16, 11, 10, 9, 8, 13, 12, 31, 30, 29, 28, 3, 4, 5, 6, 7, 2, 1, 23, 24, 35, 27, 0
\item $v = 23$: 26, 4, 3, 28, 27, 35, 24, 25, 20, 19, 18, 17, 22, 21, 34, 33, 32, 31, 12, 13, 14, 15, 16, 11, 10, 5, 6, 29, 30, 9, 8, 7, 2, 1, 0
\item $v = 24$: 26, 25, 20, 19, 14, 15, 32, 33, 18, 17, 16, 11, 10, 9, 8, 13, 12, 31, 30, 29, 28, 3, 4, 5, 6, 7, 2, 1, 23, 22, 21, 34, 35, 27, 0
\item $v = 25$: 26, 4, 3, 2, 1, 23, 24, 35, 34, 33, 18, 17, 22, 21, 20, 19, 14, 13, 12, 31, 32, 15, 16, 11, 10, 5, 6, 7, 8, 9, 30, 29, 28, 27, 0
\item $v = 26$: 12, 13, 8, 7, 2, 3, 4, 5, 6, 29, 28, 27, 0, 1, 23, 22, 21, 34, 35, 24, 25, 20, 19, 14, 15, 16, 17, 18, 33, 32, 31, 30, 9, 10, 11
\item $v = 27$: 26, 4, 5, 6, 7, 2, 3, 28, 29, 30, 31, 12, 13, 8, 9, 10, 11, 16, 17, 18, 19, 14, 15, 32, 33, 34, 35, 24, 25, 20, 21, 22, 23, 1, 0
\item $v = 28$: 12, 13, 14, 15, 32, 31, 30, 29, 6, 5, 10, 9, 8, 7, 2, 3, 4, 26, 25, 24, 23, 1, 0, 27, 35, 34, 33, 18, 19, 20, 21, 22, 17, 16, 11
\item $v = 29$: 26, 4, 5, 6, 7, 8, 13, 12, 31, 30, 9, 10, 11, 16, 17, 18, 19, 14, 15, 32, 33, 34, 35, 24, 25, 20, 21, 22, 23, 1, 2, 3, 28, 27, 0
\item $v = 30$: 26, 25, 20, 19, 18, 17, 22, 21, 34, 33, 32, 31, 12, 13, 14, 15, 16, 11, 10, 9, 8, 7, 2, 3, 4, 5, 6, 29, 28, 27, 35, 24, 23, 1, 0
\item $v = 31$: 12, 13, 8, 7, 6, 5, 10, 9, 30, 29, 28, 27, 0, 1, 2, 3, 4, 26, 25, 20, 21, 22, 23, 24, 35, 34, 33, 32, 15, 14, 19, 18, 17, 16, 11
\item $v = 32$: 26, 25, 24, 23, 1, 2, 7, 6, 5, 4, 3, 28, 29, 30, 31, 12, 13, 8, 9, 10, 11, 16, 15, 14, 19, 20, 21, 22, 17, 18, 33, 34, 35, 27, 0
\item $v = 33$: 26, 4, 3, 28, 27, 35, 34, 21, 20, 25, 24, 23, 22, 17, 18, 19, 14, 13, 12, 31, 32, 15, 16, 11, 10, 5, 6, 29, 30, 9, 8, 7, 2, 1, 0
\item $v = 34$: 26, 25, 24, 35, 27, 28, 29, 6, 5, 4, 3, 2, 7, 8, 13, 12, 31, 30, 9, 10, 11, 16, 17, 18, 33, 32, 15, 14, 19, 20, 21, 22, 23, 1, 0
\item $v = 35$: 26, 4, 3, 2, 1, 23, 24, 25, 20, 19, 18, 17, 22, 21, 34, 33, 32, 31, 12, 13, 14, 15, 16, 11, 10, 5, 6, 7, 8, 9, 30, 29, 28, 27, 0}
\end{itemize}

\noindent
\textbf{Proof that $b_0c_0$ and $b_4c_4$ are suitable edges for $J_{11}$}

Figure~\ref{fig:J11} shows the flower snark $J_{11}$ and here $b_0c_0$ and $b_4c_4$ correspond to the edges $(0,32)$ and $(11, 12)$, respectively.

$(0,11), (0,12), (32,11)$ and $(32,12)$ are good in $J_{11} - \{(0,32),(11,12)\}$ due to the following hamiltonian paths, respectively:

\begin{itemize}
{\small
\item 11, 10, 5, 6, 7, 8, 9, 36, 35, 34, 33, 43, 30, 29, 28, 23, 22, 17, 16, 15, 38, 37, 12, 13, 14, 19, 18, 39, 40, 21, 20, 25, 24, 41, 42, 27, 26, 31, 32, 4, 3, 2, 1, 0
\item 12, 13, 14, 15, 16, 11, 10, 5, 4, 32, 31, 26, 25, 24, 41, 42, 27, 28, 23, 22, 17, 18, 19, 20, 21, 40, 39, 38, 37, 36, 9, 8, 7, 6, 35, 34, 3, 2, 1, 29, 30, 43, 33, 0
\item 11, 10, 5, 4, 3, 2, 1, 0, 33, 34, 35, 6, 7, 8, 9, 36, 37, 12, 13, 14, 19, 18, 17, 16, 15, 38, 39, 40, 41, 24, 25, 20, 21, 22, 23, 28, 29, 30, 43, 42, 27, 26, 31, 32
\item 12, 13, 8, 7, 2, 3, 4, 5, 6, 35, 34, 33, 0, 1, 29, 28, 23, 22, 17, 18, 39, 38, 37, 36, 9, 10, 11, 16, 15, 14, 19, 20, 21, 40, 41, 24, 25, 26, 27, 42, 43, 30, 31, 32}
\end{itemize}

$(0,32),(11,12))$ is good in $J_{11} - \{(0,32),(11,12)\}$ due to the following two disjoint paths with end-vertices $0$ and $32$, and $11$ and $12$, respectively, which together span $J_{11}$.

\begin{itemize}
{\small
\item 32, 31, 26, 25, 20, 19, 14, 13, 8, 7, 2, 1, 0
\item 12, 37, 38, 15, 16, 17, 18, 39, 40, 21, 22, 23, 24, 41, 42, 27, 28, 29, 30, 43, 33, 34, 3, 4, 5, 6, 35, 36, 9, 10, 11}
\end{itemize}

The following hamiltonian paths show that at least one of $(0,32)$ or $(11,12)$ is good in $J_{11} - \{v,(0,32),(11,12)\}$ for every $v \in V(J_{11})$.

\begin{itemize}
{\footnotesize
\item $v = 0$: 12, 13, 14, 15, 38, 37, 36, 35, 6, 5, 10, 9, 8, 7, 2, 1, 29, 30, 31, 32, 4, 3, 34, 33, 43, 42, 41, 24, 23, 28, 27, 26, 25, 20, 19, 18, 39, 40, 21, 22, 17, 16, 11
\item $v = 1$: 32, 4, 5, 6, 7, 2, 3, 34, 35, 36, 37, 12, 13, 8, 9, 10, 11, 16, 17, 18, 19, 14, 15, 38, 39, 40, 41, 24, 23, 22, 21, 20, 25, 26, 31, 30, 29, 28, 27, 42, 43, 33, 0
\item $v = 2$: 12, 37, 38, 15, 14, 13, 8, 7, 6, 5, 10, 9, 36, 35, 34, 3, 4, 32, 31, 30, 29, 1, 0, 33, 43, 42, 41, 24, 23, 28, 27, 26, 25, 20, 19, 18, 39, 40, 21, 22, 17, 16, 11
\item $v = 3$: 32, 4, 5, 6, 7, 2, 1, 29, 28, 27, 26, 31, 30, 43, 42, 41, 40, 21, 22, 23, 24, 25, 20, 19, 14, 15, 38, 39, 18, 17, 16, 11, 10, 9, 8, 13, 12, 37, 36, 35, 34, 33, 0
\item $v = 4$: 32, 31, 26, 25, 24, 23, 28, 27, 42, 41, 40, 39, 18, 17, 22, 21, 20, 19, 14, 13, 12, 37, 38, 15, 16, 11, 10, 5, 6, 7, 8, 9, 36, 35, 34, 3, 2, 1, 29, 30, 43, 33, 0
\item $v = 5$: 32, 4, 3, 2, 1, 29, 28, 27, 26, 31, 30, 43, 42, 41, 40, 21, 22, 23, 24, 25, 20, 19, 14, 15, 38, 39, 18, 17, 16, 11, 10, 9, 36, 37, 12, 13, 8, 7, 6, 35, 34, 33, 0
\item $v = 6$: 32, 31, 26, 25, 24, 23, 28, 27, 42, 41, 40, 39, 18, 17, 22, 21, 20, 19, 14, 13, 12, 37, 38, 15, 16, 11, 10, 5, 4, 3, 2, 7, 8, 9, 36, 35, 34, 33, 43, 30, 29, 1, 0
\item $v = 7$: 32, 4, 5, 6, 35, 36, 37, 12, 13, 8, 9, 10, 11, 16, 17, 18, 19, 14, 15, 38, 39, 40, 41, 24, 23, 22, 21, 20, 25, 26, 31, 30, 29, 28, 27, 42, 43, 33, 34, 3, 2, 1, 0
\item $v = 8$: 32, 31, 26, 25, 24, 23, 28, 27, 42, 41, 40, 39, 18, 17, 22, 21, 20, 19, 14, 13, 12, 37, 38, 15, 16, 11, 10, 9, 36, 35, 34, 3, 4, 5, 6, 7, 2, 1, 29, 30, 43, 33, 0
\item $v = 9$: 32, 4, 3, 2, 1, 29, 28, 27, 26, 31, 30, 43, 42, 41, 40, 21, 22, 23, 24, 25, 20, 19, 14, 15, 38, 39, 18, 17, 16, 11, 10, 5, 6, 7, 8, 13, 12, 37, 36, 35, 34, 33, 0
\item $v = 10$: 12, 13, 14, 15, 38, 37, 36, 9, 8, 7, 2, 3, 34, 35, 6, 5, 4, 32, 31, 30, 29, 1, 0, 33, 43, 42, 41, 24, 23, 28, 27, 26, 25, 20, 19, 18, 39, 40, 21, 22, 17, 16, 11
\item $v = 11$: 32, 4, 3, 2, 1, 29, 28, 27, 26, 31, 30, 43, 42, 41, 40, 21, 20, 25, 24, 23, 22, 17, 16, 15, 14, 19, 18, 39, 38, 37, 12, 13, 8, 7, 6, 5, 10, 9, 36, 35, 34, 33, 0
\item $v = 12$: 32, 31, 26, 25, 24, 23, 28, 27, 42, 41, 40, 39, 18, 17, 22, 21, 20, 19, 14, 13, 8, 9, 10, 11, 16, 15, 38, 37, 36, 35, 34, 3, 4, 5, 6, 7, 2, 1, 29, 30, 43, 33, 0
\item $v = 13$: 12, 37, 36, 35, 6, 5, 10, 9, 8, 7, 2, 1, 0, 33, 34, 3, 4, 32, 31, 26, 27, 28, 29, 30, 43, 42, 41, 40, 21, 22, 23, 24, 25, 20, 19, 14, 15, 38, 39, 18, 17, 16, 11
\item $v = 14$: 32, 31, 26, 25, 24, 41, 42, 27, 28, 23, 22, 17, 18, 19, 20, 21, 40, 39, 38, 15, 16, 11, 10, 9, 8, 13, 12, 37, 36, 35, 34, 3, 4, 5, 6, 7, 2, 1, 29, 30, 43, 33, 0
\item $v = 15$: 32, 4, 3, 2, 1, 29, 28, 27, 26, 31, 30, 43, 42, 41, 40, 21, 22, 23, 24, 25, 20, 19, 14, 13, 12, 37, 38, 39, 18, 17, 16, 11, 10, 5, 6, 7, 8, 9, 36, 35, 34, 33, 0
\item $v = 16$: 12, 13, 8, 7, 2, 3, 34, 35, 6, 5, 4, 32, 31, 30, 29, 1, 0, 33, 43, 42, 41, 24, 25, 26, 27, 28, 23, 22, 17, 18, 39, 40, 21, 20, 19, 14, 15, 38, 37, 36, 9, 10, 11
\item $v = 17$: 32, 4, 3, 2, 1, 29, 28, 27, 26, 31, 30, 43, 42, 41, 40, 21, 22, 23, 24, 25, 20, 19, 18, 39, 38, 37, 12, 13, 14, 15, 16, 11, 10, 5, 6, 7, 8, 9, 36, 35, 34, 33, 0
\item $v = 18$: 32, 31, 26, 25, 24, 23, 28, 27, 42, 41, 40, 39, 38, 15, 14, 19, 20, 21, 22, 17, 16, 11, 10, 9, 8, 13, 12, 37, 36, 35, 34, 3, 4, 5, 6, 7, 2, 1, 29, 30, 43, 33, 0
\item $v = 19$: 32, 4, 3, 2, 1, 29, 28, 27, 26, 31, 30, 43, 42, 41, 40, 21, 20, 25, 24, 23, 22, 17, 18, 39, 38, 37, 12, 13, 14, 15, 16, 11, 10, 5, 6, 7, 8, 9, 36, 35, 34, 33, 0
\item $v = 20$: 32, 31, 26, 25, 24, 41, 42, 27, 28, 23, 22, 21, 40, 39, 38, 15, 14, 19, 18, 17, 16, 11, 10, 9, 8, 13, 12, 37, 36, 35, 34, 3, 4, 5, 6, 7, 2, 1, 29, 30, 43, 33, 0
\item $v = 21$: 32, 4, 3, 2, 1, 29, 28, 27, 26, 31, 30, 43, 42, 41, 40, 39, 18, 17, 22, 23, 24, 25, 20, 19, 14, 13, 12, 37, 38, 15, 16, 11, 10, 5, 6, 7, 8, 9, 36, 35, 34, 33, 0
\item $v = 22$: 32, 31, 26, 25, 24, 23, 28, 27, 42, 41, 40, 21, 20, 19, 14, 15, 38, 39, 18, 17, 16, 11, 10, 9, 8, 13, 12, 37, 36, 35, 34, 3, 4, 5, 6, 7, 2, 1, 29, 30, 43, 33, 0
\item $v = 23$: 32, 4, 3, 2, 1, 29, 28, 27, 26, 31, 30, 43, 42, 41, 24, 25, 20, 19, 18, 17, 22, 21, 40, 39, 38, 37, 12, 13, 14, 15, 16, 11, 10, 5, 6, 7, 8, 9, 36, 35, 34, 33, 0
\item $v = 24$: 32, 31, 26, 25, 20, 21, 22, 23, 28, 27, 42, 41, 40, 39, 38, 15, 14, 19, 18, 17, 16, 11, 10, 9, 8, 13, 12, 37, 36, 35, 34, 3, 4, 5, 6, 7, 2, 1, 29, 30, 43, 33, 0
\item $v = 25$: 32, 4, 3, 2, 1, 29, 28, 27, 26, 31, 30, 43, 42, 41, 24, 23, 22, 17, 18, 19, 20, 21, 40, 39, 38, 37, 12, 13, 14, 15, 16, 11, 10, 5, 6, 7, 8, 9, 36, 35, 34, 33, 0
\item $v = 26$: 32, 31, 30, 29, 1, 2, 7, 6, 5, 4, 3, 34, 35, 36, 37, 12, 13, 8, 9, 10, 11, 16, 17, 18, 19, 14, 15, 38, 39, 40, 41, 24, 25, 20, 21, 22, 23, 28, 27, 42, 43, 33, 0
\item $v = 27$: 32, 4, 3, 2, 1, 29, 28, 23, 24, 25, 26, 31, 30, 43, 42, 41, 40, 39, 18, 17, 22, 21, 20, 19, 14, 13, 12, 37, 38, 15, 16, 11, 10, 5, 6, 7, 8, 9, 36, 35, 34, 33, 0
\item $v = 28$: 32, 31, 26, 27, 42, 41, 40, 21, 22, 23, 24, 25, 20, 19, 14, 15, 38, 39, 18, 17, 16, 11, 10, 9, 8, 13, 12, 37, 36, 35, 34, 3, 4, 5, 6, 7, 2, 1, 29, 30, 43, 33, 0
\item $v = 29$: 32, 4, 3, 34, 33, 43, 30, 31, 26, 25, 24, 23, 28, 27, 42, 41, 40, 39, 18, 17, 22, 21, 20, 19, 14, 13, 12, 37, 38, 15, 16, 11, 10, 5, 6, 35, 36, 9, 8, 7, 2, 1, 0
\item $v = 30$: 32, 31, 26, 25, 20, 21, 22, 23, 24, 41, 40, 39, 38, 15, 14, 19, 18, 17, 16, 11, 10, 9, 8, 13, 12, 37, 36, 35, 34, 3, 4, 5, 6, 7, 2, 1, 29, 28, 27, 42, 43, 33, 0
\item $v = 31$: 32, 4, 3, 2, 1, 29, 30, 43, 42, 41, 24, 23, 28, 27, 26, 25, 20, 19, 18, 17, 22, 21, 40, 39, 38, 37, 12, 13, 14, 15, 16, 11, 10, 5, 6, 7, 8, 9, 36, 35, 34, 33, 0
\item $v = 32$: 12, 13, 8, 7, 2, 3, 4, 5, 6, 35, 34, 33, 0, 1, 29, 28, 27, 26, 31, 30, 43, 42, 41, 40, 21, 20, 25, 24, 23, 22, 17, 16, 15, 14, 19, 18, 39, 38, 37, 36, 9, 10, 11
\item $v = 33$: 32, 4, 5, 6, 7, 2, 3, 34, 35, 36, 37, 12, 13, 8, 9, 10, 11, 16, 17, 18, 19, 14, 15, 38, 39, 40, 41, 24, 23, 22, 21, 20, 25, 26, 31, 30, 43, 42, 27, 28, 29, 1, 0
\item $v = 34$: 12, 13, 14, 15, 38, 37, 36, 35, 6, 5, 10, 9, 8, 7, 2, 3, 4, 32, 31, 30, 29, 1, 0, 33, 43, 42, 41, 24, 23, 28, 27, 26, 25, 20, 19, 18, 39, 40, 21, 22, 17, 16, 11
\item $v = 35$: 32, 4, 5, 6, 7, 8, 13, 12, 37, 36, 9, 10, 11, 16, 17, 18, 19, 14, 15, 38, 39, 40, 41, 24, 23, 22, 21, 20, 25, 26, 31, 30, 29, 28, 27, 42, 43, 33, 34, 3, 2, 1, 0
\item $v = 36$: 32, 31, 26, 25, 24, 23, 28, 27, 42, 41, 40, 39, 18, 17, 22, 21, 20, 19, 14, 13, 12, 37, 38, 15, 16, 11, 10, 9, 8, 7, 2, 3, 4, 5, 6, 35, 34, 33, 43, 30, 29, 1, 0
\item $v = 37$: 12, 13, 8, 7, 6, 5, 10, 9, 36, 35, 34, 33, 0, 1, 2, 3, 4, 32, 31, 26, 27, 28, 29, 30, 43, 42, 41, 40, 21, 22, 23, 24, 25, 20, 19, 14, 15, 38, 39, 18, 17, 16, 11
\item $v = 38$: 32, 31, 26, 25, 24, 23, 28, 27, 42, 41, 40, 39, 18, 17, 22, 21, 20, 19, 14, 15, 16, 11, 10, 9, 8, 13, 12, 37, 36, 35, 34, 3, 4, 5, 6, 7, 2, 1, 29, 30, 43, 33, 0
\item $v = 39$: 32, 4, 3, 2, 1, 29, 28, 27, 26, 31, 30, 43, 42, 41, 40, 21, 20, 25, 24, 23, 22, 17, 18, 19, 14, 13, 12, 37, 38, 15, 16, 11, 10, 5, 6, 7, 8, 9, 36, 35, 34, 33, 0
\item $v = 40$: 32, 31, 26, 25, 24, 41, 42, 27, 28, 23, 22, 21, 20, 19, 14, 15, 38, 39, 18, 17, 16, 11, 10, 9, 8, 13, 12, 37, 36, 35, 34, 3, 4, 5, 6, 7, 2, 1, 29, 30, 43, 33, 0
\item $v = 41$: 32, 4, 3, 2, 1, 29, 28, 27, 42, 43, 30, 31, 26, 25, 24, 23, 22, 17, 18, 19, 20, 21, 40, 39, 38, 37, 12, 13, 14, 15, 16, 11, 10, 5, 6, 7, 8, 9, 36, 35, 34, 33, 0
\item $v = 42$: 32, 31, 26, 27, 28, 23, 22, 21, 20, 25, 24, 41, 40, 39, 38, 15, 14, 19, 18, 17, 16, 11, 10, 9, 8, 13, 12, 37, 36, 35, 34, 3, 4, 5, 6, 7, 2, 1, 29, 30, 43, 33, 0
\item $v = 43$: 32, 4, 3, 2, 1, 29, 30, 31, 26, 25, 24, 23, 28, 27, 42, 41, 40, 39, 18, 17, 22, 21, 20, 19, 14, 13, 12, 37, 38, 15, 16, 11, 10, 5, 6, 7, 8, 9, 36, 35, 34, 33, 0
}
\end{itemize}

\noindent
\textbf{Proof that $b_0c_0$ and $b_4c_4$ are suitable edges for $J_{13}$}

Figure~\ref{fig:J13} shows the flower snark $J_{13}$ and here $b_0c_0$ and $b_4c_4$ correspond to the edges $(0,38)$ and $(11, 12)$, respectively.

$(0,11), (0,12), (38,11)$ and $(38,12)$ are good in $J_{13} - \{(0,38),(11,12)\}$ due to the following hamiltonian paths, respectively:

\begin{itemize}
{\small
\item 11, 10, 5, 6, 7, 8, 9, 42, 41, 40, 39, 51, 36, 35, 34, 29, 28, 23, 22, 17, 16, 15, 44, 43, 12, 13, 14, 19, 18, 45, 46, 21, 20, 25, 24, 47, 48, 27, 26, 31, 30, 49, 50, 33, 32, 37, 38, 4, 3, 2, 1, 0
\item 12, 13, 14, 15, 16, 11, 10, 5, 4, 38, 37, 32, 31, 30, 29, 34, 33, 50, 49, 48, 47, 24, 23, 28, 27, 26, 25, 20, 19, 18, 17, 22, 21, 46, 45, 44, 43, 42, 9, 8, 7, 6, 41, 40, 3, 2, 1, 35, 36, 51, 39, 0
\item 11, 10, 5, 4, 3, 2, 1, 0, 39, 40, 41, 6, 7, 8, 9, 42, 43, 12, 13, 14, 19, 18, 17, 16, 15, 44, 45, 46, 47, 24, 23, 22, 21, 20, 25, 26, 31, 30, 29, 28, 27, 48, 49, 50, 51, 36, 35, 34, 33, 32, 37, 38
\item 12, 13, 8, 7, 2, 3, 4, 5, 6, 41, 40, 39, 0, 1, 35, 34, 29, 28, 23, 22, 17, 18, 45, 44, 43, 42, 9, 10, 11, 16, 15, 14, 19, 20, 21, 46, 47, 24, 25, 26, 27, 48, 49, 30, 31, 32, 33, 50, 51, 36, 37, 38}
\end{itemize}

$(0,38),(11,12))$ is good in $J_{13} - \{(0,38),(11,12)\}$ due to the following two disjoint paths with end-vertices $0$ and $38$, and $11$ and $12$, respectively, which together span $J_{13}$.

\begin{itemize}
{\small
\item 38, 37, 32, 31, 26, 25, 20, 19, 14, 13, 8, 7, 2, 1, 0
\item 12, 43, 44, 15, 16, 17, 18, 45, 46, 21, 22, 23, 24, 47, 48, 27, 28, 29, 30, 49, 50, 33, 34, 35, 36, 51, 39, 40, 3, 4, 5, 6, 41, 42, 9, 10, 11}
\end{itemize}

The following hamiltonian paths show that at least one of $(0,38)$ or $(11,12)$ is good in $J_{13} - \{v,(0,38),(11,12)\}$ for every $v \in V(J_{13})$.

\begin{itemize}
{\footnotesize
\item $v = 0$: 12, 13, 14, 15, 44, 43, 42, 41, 6, 5, 10, 9, 8, 7, 2, 1, 35, 36, 37, 38, 4, 3, 40, 39, 51, 50, 49, 30, 29, 34, 33, 32, 31, 26, 25, 24, 23, 28, 27, 48, 47, 46, 45, 18, 19, 20, 21, 22, 17, 16, 11
\item $v = 1$: 38, 4, 5, 6, 7, 2, 3, 40, 41, 42, 43, 12, 13, 8, 9, 10, 11, 16, 17, 18, 19, 14, 15, 44, 45, 46, 47, 24, 23, 22, 21, 20, 25, 26, 31, 30, 49, 48, 27, 28, 29, 34, 35, 36, 37, 32, 33, 50, 51, 39, 0
\item $v = 2$: 12, 43, 44, 15, 14, 13, 8, 7, 6, 5, 10, 9, 42, 41, 40, 3, 4, 38, 37, 36, 35, 1, 0, 39, 51, 50, 49, 30, 29, 34, 33, 32, 31, 26, 25, 24, 23, 28, 27, 48, 47, 46, 45, 18, 19, 20, 21, 22, 17, 16, 11
\item $v = 3$: 38, 4, 5, 6, 7, 2, 1, 35, 34, 33, 32, 37, 36, 51, 50, 49, 48, 27, 26, 31, 30, 29, 28, 23, 22, 21, 20, 25, 24, 47, 46, 45, 44, 15, 14, 19, 18, 17, 16, 11, 10, 9, 8, 13, 12, 43, 42, 41, 40, 39, 0
\item $v = 4$: 38, 37, 32, 31, 30, 29, 34, 33, 50, 49, 48, 47, 24, 23, 28, 27, 26, 25, 20, 19, 18, 17, 22, 21, 46, 45, 44, 43, 12, 13, 14, 15, 16, 11, 10, 5, 6, 7, 8, 9, 42, 41, 40, 3, 2, 1, 35, 36, 51, 39, 0
\item $v = 5$: 38, 4, 3, 2, 1, 35, 34, 33, 32, 37, 36, 51, 50, 49, 48, 27, 26, 31, 30, 29, 28, 23, 22, 21, 20, 25, 24, 47, 46, 45, 44, 15, 14, 19, 18, 17, 16, 11, 10, 9, 42, 43, 12, 13, 8, 7, 6, 41, 40, 39, 0
\item $v = 6$: 38, 37, 32, 31, 30, 29, 34, 33, 50, 49, 48, 47, 24, 23, 28, 27, 26, 25, 20, 19, 18, 17, 22, 21, 46, 45, 44, 43, 12, 13, 14, 15, 16, 11, 10, 5, 4, 3, 2, 7, 8, 9, 42, 41, 40, 39, 51, 36, 35, 1, 0
\item $v = 7$: 38, 4, 5, 6, 41, 42, 43, 12, 13, 8, 9, 10, 11, 16, 17, 18, 19, 14, 15, 44, 45, 46, 47, 24, 23, 22, 21, 20, 25, 26, 31, 30, 29, 28, 27, 48, 49, 50, 51, 36, 37, 32, 33, 34, 35, 1, 2, 3, 40, 39, 0
\item $v = 8$: 38, 37, 32, 31, 30, 29, 34, 33, 50, 49, 48, 47, 24, 23, 28, 27, 26, 25, 20, 19, 18, 17, 22, 21, 46, 45, 44, 43, 12, 13, 14, 15, 16, 11, 10, 9, 42, 41, 40, 3, 4, 5, 6, 7, 2, 1, 35, 36, 51, 39, 0
\item $v = 9$: 38, 4, 3, 2, 1, 35, 34, 33, 32, 37, 36, 51, 50, 49, 48, 27, 26, 31, 30, 29, 28, 23, 22, 21, 20, 25, 24, 47, 46, 45, 44, 15, 14, 19, 18, 17, 16, 11, 10, 5, 6, 7, 8, 13, 12, 43, 42, 41, 40, 39, 0
\item $v = 10$: 12, 13, 14, 15, 44, 43, 42, 9, 8, 7, 2, 3, 40, 41, 6, 5, 4, 38, 37, 36, 35, 1, 0, 39, 51, 50, 49, 30, 29, 34, 33, 32, 31, 26, 25, 24, 23, 28, 27, 48, 47, 46, 45, 18, 19, 20, 21, 22, 17, 16, 11
\item $v = 11$: 38, 4, 3, 2, 1, 35, 34, 33, 32, 37, 36, 51, 50, 49, 48, 27, 26, 31, 30, 29, 28, 23, 22, 21, 46, 47, 24, 25, 20, 19, 14, 15, 16, 17, 18, 45, 44, 43, 12, 13, 8, 7, 6, 5, 10, 9, 42, 41, 40, 39, 0
\item $v = 12$: 38, 37, 32, 31, 30, 29, 34, 33, 50, 49, 48, 47, 24, 23, 28, 27, 26, 25, 20, 19, 18, 17, 22, 21, 46, 45, 44, 43, 42, 9, 8, 13, 14, 15, 16, 11, 10, 5, 4, 3, 2, 7, 6, 41, 40, 39, 51, 36, 35, 1, 0
\item $v = 13$: 12, 43, 42, 41, 6, 5, 10, 9, 8, 7, 2, 1, 0, 39, 40, 3, 4, 38, 37, 32, 33, 34, 35, 36, 51, 50, 49, 48, 27, 26, 31, 30, 29, 28, 23, 22, 21, 20, 25, 24, 47, 46, 45, 44, 15, 14, 19, 18, 17, 16, 11
\item $v = 14$: 38, 37, 32, 31, 30, 29, 34, 33, 50, 49, 48, 47, 24, 23, 28, 27, 26, 25, 20, 19, 18, 17, 22, 21, 46, 45, 44, 15, 16, 11, 10, 9, 8, 13, 12, 43, 42, 41, 40, 3, 4, 5, 6, 7, 2, 1, 35, 36, 51, 39, 0
\item $v = 15$: 38, 4, 3, 2, 1, 35, 34, 33, 32, 37, 36, 51, 50, 49, 48, 27, 26, 31, 30, 29, 28, 23, 22, 21, 20, 25, 24, 47, 46, 45, 44, 43, 12, 13, 14, 19, 18, 17, 16, 11, 10, 5, 6, 7, 8, 9, 42, 41, 40, 39, 0
\item $v = 16$: 12, 13, 8, 7, 2, 3, 40, 41, 6, 5, 4, 38, 37, 36, 35, 1, 0, 39, 51, 50, 49, 30, 29, 34, 33, 32, 31, 26, 25, 24, 23, 28, 27, 48, 47, 46, 45, 18, 17, 22, 21, 20, 19, 14, 15, 44, 43, 42, 9, 10, 11
\item $v = 17$: 38, 4, 3, 2, 1, 35, 34, 33, 32, 37, 36, 51, 50, 49, 48, 27, 26, 31, 30, 29, 28, 23, 22, 21, 20, 25, 24, 47, 46, 45, 18, 19, 14, 13, 12, 43, 44, 15, 16, 11, 10, 5, 6, 7, 8, 9, 42, 41, 40, 39, 0
\item $v = 18$: 38, 37, 32, 31, 30, 29, 34, 33, 50, 49, 48, 47, 24, 23, 28, 27, 26, 25, 20, 19, 14, 15, 44, 45, 46, 21, 22, 17, 16, 11, 10, 9, 8, 13, 12, 43, 42, 41, 40, 3, 4, 5, 6, 7, 2, 1, 35, 36, 51, 39, 0
\item $v = 19$: 38, 4, 3, 2, 1, 35, 34, 33, 32, 37, 36, 51, 50, 49, 48, 27, 28, 29, 30, 31, 26, 25, 20, 21, 46, 47, 24, 23, 22, 17, 18, 45, 44, 43, 12, 13, 14, 15, 16, 11, 10, 5, 6, 7, 8, 9, 42, 41, 40, 39, 0
\item $v = 20$: 38, 37, 32, 31, 30, 29, 34, 33, 50, 49, 48, 47, 24, 25, 26, 27, 28, 23, 22, 21, 46, 45, 44, 15, 14, 19, 18, 17, 16, 11, 10, 9, 8, 13, 12, 43, 42, 41, 40, 3, 4, 5, 6, 7, 2, 1, 35, 36, 51, 39, 0
\item $v = 21$: 38, 4, 3, 2, 1, 35, 34, 33, 32, 37, 36, 51, 50, 49, 48, 27, 26, 31, 30, 29, 28, 23, 22, 17, 18, 19, 20, 25, 24, 47, 46, 45, 44, 43, 12, 13, 14, 15, 16, 11, 10, 5, 6, 7, 8, 9, 42, 41, 40, 39, 0
\item $v = 22$: 38, 37, 32, 31, 30, 29, 34, 33, 50, 49, 48, 47, 24, 23, 28, 27, 26, 25, 20, 21, 46, 45, 44, 15, 14, 19, 18, 17, 16, 11, 10, 9, 8, 13, 12, 43, 42, 41, 40, 3, 4, 5, 6, 7, 2, 1, 35, 36, 51, 39, 0
\item $v = 23$: 38, 4, 3, 2, 1, 35, 34, 33, 32, 37, 36, 51, 50, 49, 48, 27, 28, 29, 30, 31, 26, 25, 24, 47, 46, 45, 18, 17, 22, 21, 20, 19, 14, 13, 12, 43, 44, 15, 16, 11, 10, 5, 6, 7, 8, 9, 42, 41, 40, 39, 0
\item $v = 24$: 38, 37, 32, 31, 30, 29, 34, 33, 50, 49, 48, 47, 46, 21, 22, 23, 28, 27, 26, 25, 20, 19, 14, 15, 44, 45, 18, 17, 16, 11, 10, 9, 8, 13, 12, 43, 42, 41, 40, 3, 4, 5, 6, 7, 2, 1, 35, 36, 51, 39, 0
\item $v = 25$: 38, 4, 3, 2, 1, 35, 34, 33, 32, 37, 36, 51, 50, 49, 48, 27, 26, 31, 30, 29, 28, 23, 24, 47, 46, 45, 18, 17, 22, 21, 20, 19, 14, 13, 12, 43, 44, 15, 16, 11, 10, 5, 6, 7, 8, 9, 42, 41, 40, 39, 0
\item $v = 26$: 38, 37, 32, 31, 30, 29, 34, 33, 50, 49, 48, 27, 28, 23, 22, 21, 20, 25, 24, 47, 46, 45, 44, 15, 14, 19, 18, 17, 16, 11, 10, 9, 8, 13, 12, 43, 42, 41, 40, 3, 4, 5, 6, 7, 2, 1, 35, 36, 51, 39, 0
\item $v = 27$: 38, 4, 3, 2, 1, 35, 34, 33, 32, 37, 36, 51, 50, 49, 48, 47, 24, 23, 28, 29, 30, 31, 26, 25, 20, 19, 18, 17, 22, 21, 46, 45, 44, 43, 12, 13, 14, 15, 16, 11, 10, 5, 6, 7, 8, 9, 42, 41, 40, 39, 0
\item $v = 28$: 38, 37, 32, 31, 30, 29, 34, 33, 50, 49, 48, 27, 26, 25, 20, 21, 22, 23, 24, 47, 46, 45, 44, 15, 14, 19, 18, 17, 16, 11, 10, 9, 8, 13, 12, 43, 42, 41, 40, 3, 4, 5, 6, 7, 2, 1, 35, 36, 51, 39, 0
\item $v = 29$: 38, 4, 3, 2, 1, 35, 34, 33, 32, 37, 36, 51, 50, 49, 30, 31, 26, 25, 24, 23, 28, 27, 48, 47, 46, 45, 18, 17, 22, 21, 20, 19, 14, 13, 12, 43, 44, 15, 16, 11, 10, 5, 6, 7, 8, 9, 42, 41, 40, 39, 0
\item $v = 30$: 38, 37, 32, 31, 26, 27, 28, 29, 34, 33, 50, 49, 48, 47, 46, 21, 22, 23, 24, 25, 20, 19, 14, 15, 44, 45, 18, 17, 16, 11, 10, 9, 8, 13, 12, 43, 42, 41, 40, 3, 4, 5, 6, 7, 2, 1, 35, 36, 51, 39, 0
\item $v = 31$: 38, 4, 3, 2, 1, 35, 34, 33, 32, 37, 36, 51, 50, 49, 30, 29, 28, 23, 24, 25, 26, 27, 48, 47, 46, 45, 18, 17, 22, 21, 20, 19, 14, 13, 12, 43, 44, 15, 16, 11, 10, 5, 6, 7, 8, 9, 42, 41, 40, 39, 0
\item $v = 32$: 38, 37, 36, 35, 1, 2, 7, 6, 5, 4, 3, 40, 41, 42, 43, 12, 13, 8, 9, 10, 11, 16, 17, 18, 19, 14, 15, 44, 45, 46, 47, 24, 23, 22, 21, 20, 25, 26, 31, 30, 49, 48, 27, 28, 29, 34, 33, 50, 51, 39, 0
\item $v = 33$: 38, 4, 3, 2, 1, 35, 34, 29, 30, 31, 32, 37, 36, 51, 50, 49, 48, 47, 24, 23, 28, 27, 26, 25, 20, 19, 18, 17, 22, 21, 46, 45, 44, 43, 12, 13, 14, 15, 16, 11, 10, 5, 6, 7, 8, 9, 42, 41, 40, 39, 0
\item $v = 34$: 38, 37, 32, 33, 50, 49, 48, 27, 26, 31, 30, 29, 28, 23, 22, 21, 20, 25, 24, 47, 46, 45, 44, 15, 14, 19, 18, 17, 16, 11, 10, 9, 8, 13, 12, 43, 42, 41, 40, 3, 4, 5, 6, 7, 2, 1, 35, 36, 51, 39, 0
\item $v = 35$: 38, 4, 3, 40, 39, 51, 36, 37, 32, 31, 30, 29, 34, 33, 50, 49, 48, 47, 24, 23, 28, 27, 26, 25, 20, 19, 18, 17, 22, 21, 46, 45, 44, 43, 12, 13, 14, 15, 16, 11, 10, 5, 6, 41, 42, 9, 8, 7, 2, 1, 0
\item $v = 36$: 38, 37, 32, 31, 26, 27, 28, 29, 30, 49, 48, 47, 46, 21, 22, 23, 24, 25, 20, 19, 14, 15, 44, 45, 18, 17, 16, 11, 10, 9, 8, 13, 12, 43, 42, 41, 40, 3, 4, 5, 6, 7, 2, 1, 35, 34, 33, 50, 51, 39, 0
\item $v = 37$: 38, 4, 3, 2, 1, 35, 36, 51, 50, 49, 30, 29, 34, 33, 32, 31, 26, 25, 24, 23, 28, 27, 48, 47, 46, 45, 18, 17, 22, 21, 20, 19, 14, 13, 12, 43, 44, 15, 16, 11, 10, 5, 6, 7, 8, 9, 42, 41, 40, 39, 0
\item $v = 38$: 12, 13, 8, 7, 2, 3, 4, 5, 6, 41, 40, 39, 0, 1, 35, 34, 33, 32, 37, 36, 51, 50, 49, 48, 27, 26, 31, 30, 29, 28, 23, 22, 21, 46, 47, 24, 25, 20, 19, 14, 15, 16, 17, 18, 45, 44, 43, 42, 9, 10, 11
\item $v = 39$: 38, 4, 5, 6, 7, 2, 3, 40, 41, 42, 43, 12, 13, 8, 9, 10, 11, 16, 17, 18, 19, 14, 15, 44, 45, 46, 47, 24, 23, 22, 21, 20, 25, 26, 31, 30, 29, 28, 27, 48, 49, 50, 51, 36, 37, 32, 33, 34, 35, 1, 0
\item $v = 40$: 12, 13, 14, 15, 44, 43, 42, 41, 6, 5, 10, 9, 8, 7, 2, 3, 4, 38, 37, 36, 35, 1, 0, 39, 51, 50, 49, 30, 29, 34, 33, 32, 31, 26, 25, 24, 23, 28, 27, 48, 47, 46, 45, 18, 19, 20, 21, 22, 17, 16, 11
\item $v = 41$: 38, 4, 5, 6, 7, 8, 13, 12, 43, 42, 9, 10, 11, 16, 17, 18, 19, 14, 15, 44, 45, 46, 47, 24, 23, 22, 21, 20, 25, 26, 31, 30, 29, 28, 27, 48, 49, 50, 51, 36, 37, 32, 33, 34, 35, 1, 2, 3, 40, 39, 0
\item $v = 42$: 38, 37, 32, 31, 30, 29, 34, 33, 50, 49, 48, 47, 24, 23, 28, 27, 26, 25, 20, 19, 18, 17, 22, 21, 46, 45, 44, 43, 12, 13, 14, 15, 16, 11, 10, 9, 8, 7, 2, 3, 4, 5, 6, 41, 40, 39, 51, 36, 35, 1, 0
\item $v = 43$: 12, 13, 8, 7, 6, 5, 10, 9, 42, 41, 40, 39, 0, 1, 2, 3, 4, 38, 37, 32, 33, 34, 35, 36, 51, 50, 49, 48, 27, 26, 31, 30, 29, 28, 23, 22, 21, 20, 25, 24, 47, 46, 45, 44, 15, 14, 19, 18, 17, 16, 11
\item $v = 44$: 38, 37, 32, 31, 30, 29, 34, 33, 50, 49, 48, 47, 24, 25, 26, 27, 28, 23, 22, 17, 18, 45, 46, 21, 20, 19, 14, 15, 16, 11, 10, 9, 8, 13, 12, 43, 42, 41, 40, 3, 4, 5, 6, 7, 2, 1, 35, 36, 51, 39, 0
\item $v = 45$: 38, 4, 3, 2, 1, 35, 34, 33, 32, 37, 36, 51, 50, 49, 48, 27, 28, 29, 30, 31, 26, 25, 20, 21, 46, 47, 24, 23, 22, 17, 18, 19, 14, 13, 12, 43, 44, 15, 16, 11, 10, 5, 6, 7, 8, 9, 42, 41, 40, 39, 0
\item $v = 46$: 38, 37, 32, 31, 30, 29, 34, 33, 50, 49, 48, 47, 24, 25, 26, 27, 28, 23, 22, 21, 20, 19, 14, 15, 44, 45, 18, 17, 16, 11, 10, 9, 8, 13, 12, 43, 42, 41, 40, 3, 4, 5, 6, 7, 2, 1, 35, 36, 51, 39, 0
\item $v = 47$: 38, 4, 3, 2, 1, 35, 34, 33, 32, 37, 36, 51, 50, 49, 48, 27, 26, 31, 30, 29, 28, 23, 24, 25, 20, 19, 18, 17, 22, 21, 46, 45, 44, 43, 12, 13, 14, 15, 16, 11, 10, 5, 6, 7, 8, 9, 42, 41, 40, 39, 0
\item $v = 48$: 38, 37, 32, 31, 30, 49, 50, 33, 34, 29, 28, 27, 26, 25, 20, 21, 22, 23, 24, 47, 46, 45, 44, 15, 14, 19, 18, 17, 16, 11, 10, 9, 8, 13, 12, 43, 42, 41, 40, 3, 4, 5, 6, 7, 2, 1, 35, 36, 51, 39, 0
\item $v = 49$: 38, 4, 3, 2, 1, 35, 34, 33, 50, 51, 36, 37, 32, 31, 30, 29, 28, 23, 24, 25, 26, 27, 48, 47, 46, 45, 18, 17, 22, 21, 20, 19, 14, 13, 12, 43, 44, 15, 16, 11, 10, 5, 6, 7, 8, 9, 42, 41, 40, 39, 0
\item $v = 50$: 38, 37, 32, 33, 34, 29, 28, 27, 26, 31, 30, 49, 48, 47, 46, 21, 22, 23, 24, 25, 20, 19, 14, 15, 44, 45, 18, 17, 16, 11, 10, 9, 8, 13, 12, 43, 42, 41, 40, 3, 4, 5, 6, 7, 2, 1, 35, 36, 51, 39, 0
\item $v = 51$: 38, 4, 3, 2, 1, 35, 36, 37, 32, 31, 30, 29, 34, 33, 50, 49, 48, 47, 24, 23, 28, 27, 26, 25, 20, 19, 18, 17, 22, 21, 46, 45, 44, 43, 12, 13, 14, 15, 16, 11, 10, 5, 6, 7, 8, 9, 42, 41, 40, 39, 0
}
\end{itemize}

\end{document}